\documentclass[11pt]{article}
\usepackage[dvips]{graphicx}
\usepackage{color}
\usepackage{amsmath}
\usepackage{amssymb}
\usepackage{amsthm}
\usepackage{latexsym}
\usepackage{authblk}
\usepackage{bm}

\def\sfd{{\sf d}}
\def\sfp{{\sf p}}

\usepackage{comment}
\includecomment{jp-text}
\excludecomment{en-text}
\def\be{\begin{equation}}
\def\ee{\end{equation}}
\def\bea{\begin{eqnarray}}
\def\eea{\end{eqnarray}}
\def\beas{\begin{eqnarray*}}
\def\eeas{\end{eqnarray*}}

\def\im{\item}
\def\calm{{\cal m}}
\def\half{{\frac{1}{2}}}

\newif\ifcol
\colfalse

\ifcol
\newcommand{\blue}{\color[rgb]{0,0,0.8}}
\newcommand{\colorc}{\color{cyan}}
\newcommand{\colorr}{\color[rgb]{0.8,0,0}}
\newcommand{\colorg}{\color[rgb]{0,0.5,0}}
\newcommand{\colorb}{\color[rgb]{0,0,0.8}}

\newcommand{\colorn}{\color[rgb]{1,1,1}}

\newcommand{\coloro}{\color[rgb]{1,0.851,0}}
\newcommand{\coloroy}{\color[rgb]{1,0.95,0}}

\newcommand{\colorrz}{\color[rgb]{0.8,0,0}}
\else
\newcommand{\blue}{\color{black}}
\newcommand{\colorc}{\color{black}}
\newcommand{\colorr}{\color{black}}
\newcommand{\colorg}{\color{black}}
\newcommand{\colorb}{\color{black}}

\newcommand{\colorn}{\color{black}}

\newcommand{\coloro}{\color{black}}
\newcommand{\coloroy}{\color{black}}

\newcommand{\colorrz}{\color{black}}
\fi

\def\koko{{\coloroy{koko}}}
\def\D2{\bbD_{2,\infty-}}

\usepackage{mathrsfs} 
\usepackage{graphicx} %
\usepackage{layout} %
\newtheorem{theorem*}{Theorem}[section]

\newtheorem{note*}[theorem*]{Note}
\newtheorem{lemma*}[theorem*]{Lemma}
\newtheorem{definition*}[theorem*]{Definition}
\newtheorem{proposition*}[theorem*]{Proposition}

\newtheorem{remark*}[theorem*]{Remark}
\newtheorem{example*}[theorem*]{Example}
\numberwithin{equation}{section}

\def\D{{\bf D}}

\def\F{{\bf F}}

\def\calb{{\cal B}}

\def\cald{{\cal D}}

\def\calf{{\cal F}}

\def\cali{{\cal I}}

\def\calm{{\cal M}}

\def\calx{{\cal X}}

\def\sskip{\hspace{0.5cm}}

\def\ep{\epsilon}
\def\half{\frac{1}{2}}

\def\Iku{\Rightarrow}

\def\up{\uparrow}
\def\down{\downarrow}

\def\nn{\nonumber}
\def\be{\begin{equation}}
\def\ee{\end{equation}}
\def\bea{\begin{eqnarray}}
\def\eea{\end{eqnarray}}
\def\beas{\begin{eqnarray*}}
\def\eeas{\end{eqnarray*}}
\def\odg{{\dot{\lambda}}}
\def\dodg{{\dot{g}}}
\def\ddodg{{\ddot{g}}}

\def\l{\left}
\def\r{\right}

\newtheorem{remark}{Remark}

\def\D{{\bf D}}

\def\F{{\bf F}}

\def\calb{{\cal B}}

\def\cald{{\cal D}}

\def\calf{{\cal F}}

\def\cali{{\cal I}}

\def\calm{{\cal M}}

\def\calx{{\cal X}}

\def\sskip{\hspace{0.5cm}}

\def\ep{\epsilon}
\def\half{\frac{1}{2}}

\def\Iku{\Rightarrow}

\def\up{\uparrow}
\def\down{\downarrow}

\def\halflineskip{\vspace*{3mm}}
\def\nn{\nonumber}
\def\be{\begin{equation}}
\def\ee{\end{equation}}
\def\bea{\begin{eqnarray}}
\def\eea{\end{eqnarray}}
\def\beas{\begin{eqnarray*}}
\def\eeas{\end{eqnarray*}}
\def\bi{\begin{itemize}}
\def\ei{\end{itemize}}
\def\im{\item}
\def\bd{\begin{description}}
\def\ed{\end{description}}
\def\l{\left}
\def\r{\right}

\newcommand{\bbB}{{\mathbb B}}

\newcommand{\bbD}{{\mathbb D}}
\newcommand{\bbE}{{\mathbb E}}

\newcommand{\bbH}{{\mathbb H}}

\newcommand{\bbN}{{\mathbb N}}

\newcommand{\bbP}{{\mathbb P}}

\newcommand{\bbR}{{\mathbb R}}

\newcommand{\bbU}{{\mathbb U}}
\newcommand{\bbV}{{\mathbb V}}

\newcommand{\bbX}{{\mathbb X}}
\newcommand{\bbY}{{\mathbb Y}}
\newcommand{\bbZ}{{\mathbb Z}}

\begin{document}
\title{
Quasi likelihood analysis of point processes for ultra high frequency data
\footnote{
This work was in part supported by 
Japan Society for the Promotion of Science Grants-in-Aid for Scientific Research Nos. 24340015 
(Scientific Research), 
Nos. 24650148 and 26540011 (Challenging Exploratory Research); 
CREST Japan Science and Technology Agency; 
NS Solutions Corporation; 
and by a Cooperative Research Program of the Institute of Statistical Mathematics.
}
}

\author[1,2,4]{Teppei Ogihara}
\author[3,4,1]{Nakahiro Yoshida}
\affil[1]{Institue of Statistical Mathematics\footnote
        {Institue of Statistical Mathematics: 10-3 Midori-cho, Tachikawa, Tokyo 190-8562, Japan}
        }
\affil[2]{SOKENDAI (The Graduate University 
for Advanced Studies)\footnote
        {Shonan Village, Hayama, Kanagawa 240-0193, Japan}
        }
\affil[3]{Graduate School of Mathematical Sciences, University of Tokyo\footnote
        {Graduate School of Mathematical Sciences, University of Tokyo: 3-8-1 Komaba, Meguro-ku, Tokyo 153-8914, Japan}
        }
\affil[4]{CREST, Japan Science and Technology Agency
        }

\maketitle
\noindent

\noindent
\begin{abstract}\noindent
We introduce a point process regression model that is applicable to 
price models and limit order book models. 
Hawkes type autoregression in the intensity process is generalized to a 
stochastic regression to covariate processes. 
We establish the so-called 
quasi likelihood analysis, which gives a polynomial type large deviation estimate  
for the statistical random field. 
We derive large sample properties of the maximum likelihood type estimator and 
the Bayesian type estimator when the intensity processes become large 
under a finite time horizon. There appears non-ergodic statistics.
A classical approach is also mentioned.

\end{abstract}
\begin{en-text}
We consider a stochastic regression model defined by a stochastic differential equation. 
Based on an expected Kullback-Leibler information for the predictive spot volatility,
we propose and validate an information criterion $sVIC_x$ for selection of volatility models. 
Some variants of $sVIC_x$ are also discussed. 
We give examples and simulation results of 
model selection based on the information criterion.
\end{en-text}

\vspace{5pt} 



\vspace{5pt} 

\noindent



\section{Introduction}

High-frequency financial data is one of the latest objects to
be challenged by the most advanced statistics. Ad hoc
descriptive methods often mislead and fail data analysis. For
two stochastic processes observed high-frequently and
asynchronously, the estimated covariance between them almost
vanishes if one applies the realized covariance with a ``natural''
interpolation method, even if the real covariance is not null.
Recently it was recognized that the non-synchronicity of
sampling schemes, though it is inevitable for real financial
data, causes such phenomena generically called the Epps effect.
Theory of non-synchronous estimation has been dramatically
developing for the last decade and successfully applied to
actual data analyses.
Market microstructure is another main factor that causes
the Epps effect. Remarkable progresses were recently made in
volatility estimation problems by proposing effective filters
that remove microstructure noises and at the same time treat
non-synchronous sampling schemes. Non-synchronicity and
microstructure are now the point where theoretical statistics,
probability theory and real data analysis are confluent 
{\blue (
Epps \cite{Epps1979}, 
Malliavin and Mancino \cite{MalliavinMancino2002,malliavin2009fourier}, 
Hayashi and Yoshida \cite{HayashiYoshida2005,HayashiYoshida2008,
HayashiYoshida2011}, 
Voev and Lunde \cite{voev2007integrated}, 
Griffin and Oomen \cite{griffin2011covariance}, 
Mykland \cite{mykland2012gaussian}, 
%
Zhou \cite{zhou1996high}, 
Zhang et al. \cite{zhang2005tale}, 
Zhang \cite{zhang2006efficient}, 
Podolskij and Better \cite{podolskij2009estimation},
Jacod et al. \cite{jacod2009microstructure}
%
Christensen et al. \cite{christensen2010pre}, 
Bibinger \cite{bibinger2011efficient,bibinger2012estimator}, 
Ogihara and Yoshida \cite{ogihara2014quasi}, 
Koike \cite{koike2013estimation, koike2014limit,koike2014estimator},  
Ogihara \cite{ogihara2015local,ogihara2014parametric} 
among many others). 
}

{\blue The} very latest issue is modeling of ultra high frequency
phenomena by point processes. It enables us to model
microstructure itself rather than eliminating it as noise. In
ultra high frequency sampling, the central limit theorem does
not work and there is no longer Brownian motion as the driving factor of asset prices, differently
from the standard mathematical finance. The world of real data
is already beyond a standard theory but this is the reality
statisticians are confronted with.

Statistical theory of non-synchronous data suggests
relativity of time. Theory of lead-lag estimation emerged
against this background. When observing two time series, we
often find lead-lag between them, namely, one is the leader
and the other is the follower. If these series are stock prices,
this means the existence of {\blue statistical} arbitrage. 
\begin{en-text}
This fact contradicts
the principle of mathematical finance, however it is commonly recognized 
in the financial world today. 
\end{en-text}
{\blue Some developments for high frequency data are in 
de Jong and Nijman \cite{Jong1997}, 
\cite{HoffmannRosenbaumYoshida2013} and Abergel and Huss \cite{abergel2012high}.}

In this article, we consider a point process regression model that enables us
to express nonsynchronicity of observations, lead-lag relation and microstructure.
Our model can describe self-exciting/self-correcting effects of the point
processes as well as exogenous effects. Non-ergodic statistics is constructed
in the QLA (quasi likelihood analysis) framework. The point process regression
model has applications to price models and limit order book models.

\begin{en-text}
In the present article, we introduce a point process model that enables us to 
express asynchronicity of the sampling and microstructure. 
Our model describes self-exciting/self-correcting effects of the point processes 
as well as exogenous effects. 
The lead-lag structure is also incorporated in the model, and 
we will propose an estimation method of the lead-lag. 
\end{en-text}

\section{Point process regression model}

\begin{en-text}
Let $\cali=\{1,...,\sfd\}$. 
The multi-variate point process $N^n=(\bbN^{n,\alpha})_{\alpha\in\cali}$ is assumed to have a  $$\sfd$-dimensional intensity process 
$n\lambda^n(t,\theta^*,\gamma^*,X^n)
=(n\lambda^{n,\alpha}(t,\theta^*,\gamma^*,\bbN))_{\alpha\in\cali}$ 
defined by 
\beas 
\lambda^{n,\alpha}(t,\theta^*,\gamma^*,\bbN)
&=& 
g^{n,\alpha}(t,\gamma^*)+\int_{T_0}^{t-}
K^{n,\alpha}_\beta(t,t+\Delta_\beta^*-\Delta_\alpha^*-s,\gamma^*)dX^{n,\beta}_s,
\eeas
where Einstein's rule for repeated indices is applied to $\beta$. 
$\gamma^*$ is a parameter that specifies the true structure in a parametric model 
introduced later. 
The point process $\bbN^{n,\alpha}$ forms a model whose intensity processes refer to 
the covariates $g^{n,\alpha}$ and $K^{n,\alpha}_\beta$ but with lags. 
\end{en-text}

The $\sfd$-dimensional point process $N^n=(N^{n,\alpha})_{\alpha\in\cali}$ 
on the interval $I=[T_0,T_1]$, $\cali=\{1,...,\sfd\}$,  
is assumed to have a $\sfd$-dimensional intensity process 
$n\lambda^n(t,\theta)$ 
defined by 
\beas 
\lambda^n(t,\theta)
&=& 
g^n(t,\theta)+\int_{\hat{T}_0}^{t-}
K^n(t,s,\theta)dX^n_s,
\eeas
where 
$\theta$ is a parameter. 
The point process $N^n$ forms a model whose intensity processes refer to 
the covariates $g^n$ and 
$K^n$ as well as the explanatory process $X^n$.\footnote{It is possible to make $g^n(t,\theta)$ include 
the part $\int_{\hat{T}_0}^{T_0}K^n(t,s,\theta)dX^n_s$. However this definition fits Hawkes type processes we will discuss later. }
More precisely, we will work on a stochastic basis $\calb=(\Omega,\calf,{\bf F},P)$, 
${\bf F}=(\calf_t)_{t\in \hat{I}}$ being a filtration on $(\Omega,\calf)$, where $\hat{I}=[\hat{T}_0,T_1]
\supset I$. 
For each $n\in\bbN$ and $\theta\in\Theta$, 
$(g^n(t,\theta))_{t\in I}$ is a $\sfd$-dimensional predictable process,  
$(K^n(t,s,\theta))_{s\in[\hat{T}_0,t)}$ is a $\sfd\times\sfd_0$ matrix-valued optional process 
for $t\in I$, 
$\cali_0=\{1,...,\sfd_0\}$, 
and $(X^n_t)_{t\in\hat{I}}$ is a $\sfd_0$-dimensional ${\bf F}$-adapted right-continuous increasing {\blue (i.e., non-decreasing)} process on $\calb$. 
We will impose conditions that ensure the existence of those stochastic integrals later. 
The multivariate point process $N^n$ is compensated by the process 
$(\int_{T_0}^t n\lambda^n(s,\theta)ds)_{t\in I}$ when $\theta$ is the true value of the unknown parameter. 
{\colorb We will assume that any two elements of $N^n$ do not share common jumps.
\footnote{{\colorb This assumption is necessary to specify the asymptotic variance of 
the estimators, as will be done later.}}}

Applications of point processes to financial data were in 
Hewlett \cite{hewlett2006clustering} on the clustered arrivals of buy and sell trades using Hawkes processes, 
Large \cite{large2007measuring} on extension by using a fine description of orders, 
Bowsher \cite{bowsher2007modelling} on a generalized Hawkes model, 
and {\blue in} Bacry et al. \cite{bacry2013modelling} on a price model. 
Chen and Hall \cite{chen2013inference} investigated the maximum likelihood estimator 
of a non-stationary self-exciting point process when the intensity goes up.

Obviously, the transaction times of stocks whose occurrence intensities possibly depend on their own or exogenous randomly changing factors 
can be described by a point process regression model. 
The point process regression model also applies to micro-scale modeling of the movements of the stock prices, 
incorporating information of covariate processes. 

Recently, point processes are applied to order-book modeling; see 
Cont et al. \cite{cont2010stochastic} and Abergel and Jedidi \cite{abergel2013mathematical,  abergel2015long}, and also 
Smith et al. \cite{smith2003statistical}. 
Abergel and Jedidi \cite{abergel2013mathematical,abergel2015long} presented an order book model by a multivariate point process 
and proved ergodicity of the system in infinite time horizon by using the drift condition for the Markov chain. 
Muni Toke and Pomponio \cite{muni2011modelling} gave a model of trades-through in a limited order book using Hawkes processes. 

Our point process regression model has finite time horizon and the resulting statistics becomes non-ergodic. 
We shall give short descriptions of the last two examples, before going into the main part of this article.

\subsection{Modeling digital movements of stock prices}
The point process $N^n$ is fairly generic. 
For ultra high-frequency financial data, we model the movements of the prices $Y_t$ 
by the combination of the components of the multi-variate point process 
$N^n=(N^{n,\alpha})_{\alpha\in\cali}$, e.g., by 
\bea\label{AN} 
Y_t &=& AN^n_t,
\eea
where $A$ is a constant matrix. 
In the following examples, $N$ denotes $N^n$ for notational simplicity. 

\begin{example*}\label{one-unit}\rm ($\pm$ one-unit jumps) 
Let $a$ denote a monetary unit. 
Let 
\beas 
A &=& 
a\l[\begin{array}{cccc}
1&-1&0&0\\
0&0&1&-1
\end{array}
\r]
\eeas
Then  
\beas 
Y &=& 
\l[\begin{array}{c}
a(N^{0}-N^{1}) \\
a(N^{2}-N^{3})
\end{array}
\r]
\eeas
for $N=[N^{0},N^{1},N^{2},N^{3}]'$.
\end{example*}

\begin{example*}\label{one/two}\rm ($\pm$ one/two-unit jumps)
Let 
\beas 
A &=& 
a\l[\begin{array}{cccccccc}
1&2&-1&-2&0&0&0&0\\
0&0&0&0&1&2&-1&-2
\end{array}
\r]
\eeas
Then  
\beas 
Y &=& 
\l[\begin{array}{c}
aN^{0}+2aN^{1}-aN^{2}-2aN^{3} \\
aN^{4}+2aN^{5}-aN^{6}-2aN^{7} 
\end{array}
\r]
\eeas
for $N=[N^{0},...,N^{7}]'$.
\end{example*}

\begin{example*}\rm (simultaneous jumps) 
Let 
\beas 
A &=& 
a\l[\begin{array}{cccccc}
1&-1&0&0&1&-1\\
0&0&1&-1&\pm1&\mp1
\end{array}
\r]
\eeas
Then  
\beas 
Y &=& 
\l[\begin{array}{c}
a(N^{0}-N^{1})+a(N^{4}-N^{5}) \\
a(N^{2}-N^{3})\pm a(N^{4}-N^{5})
\end{array}
\r]
\eeas
for $N^n=[N^{0},...,N^{5}]'$.
\end{example*}

In this way, ultimately, we can assume that our data is described by a multi-variate point process.

\begin{example*}\rm 
The Hawkes type process is an example if one takes $X^n_t=n^{-1}N^n_t$. 
Our setting also includes the models with $X^n_t=n^{-1}V^n_t$ and $X^n_t=(n^{-1}N^n_t,n^{-1}V^n_t)$ 
for covariates subordinator $V^n_t$ like the cumulative volume of the trades. 
\end{example*}

\begin{example*}\rm 
An example of the kernel function is 
$K^n(t,s,\theta)=c(\theta_1,\theta_2)(t-s)^{\theta_2}e^{-\theta_1(t-s)}$. 
This extends the original Hawkes process and it will be important when we discuss 
the lead-lag estimation in our framework. 
\end{example*}

\subsection{Limit oder book}\label{270723-1}

Multivariate point processes give an approach to modeling of the limit order book. 
For simplicity, we will assume  the volumes of limit orders, market orders and 
cancellation are a common value $q$ for all prices. 
The state of the limit order book is described by the multi-dimensional process 
\beas
\bbX=((A^{\alpha})_{\alpha=1,...,k_A},(B^{\beta})_{\beta=1,...,k_B}), 
\eeas
where 
the process $A^{\alpha}$ counts the number of shares available at price $p_A^\alpha$ at time $t$ on the ask side 
and 
the process $B^{\beta}_t$ counts the number of shares available at price $p_B^\beta$ at time $t$ on the bid side. 
In this modeling, the state space of $\bbX$ is absolute or fixed.
{\blue 
The price $p_A^\alpha$ may denote a relative quoted price 
if one defines $p_A^\alpha$ as the price $\alpha$ ticks away from the best opposite quote, 
while it is also possible to consider a fixed state space.} 
The random evolution of $\bbX$ is determined by the processes 
$M^A$ counting number of arrivals of market orders on the ask side, 
$M^B$ of market orders on the bid side, 
$L^\alpha$ of limit orders at level $\alpha$ on the ask side, 
$L^\beta$ of limit orders at level $\beta$ on the bid side, 
$C^\alpha$ of cancellation at level $\alpha$ on the ask side, 
and $C^\beta$ of cancellation at level $\beta$ on the bid side. 
 %
The multivariate counting process $N^n$ consists of these counting processes. 
Here prices can be recognized as a function of $\bbX$. 

\begin{en-text}
For modeling of $C^\alpha$ and $C^\beta$, we may {\blue consider} $g^{n,\alpha}(t,\theta)$  
proportional 
to $A^\alpha$ or $B^\beta$, or more complicated mechanism. 
Non degeneracy of the intensity processes 
for validating likelihood analysis 
seems to be problematic {\blue due to the shape of the quasi likelihood,} 
but it causes no difficulty thanks to a positive minimum unit {\blue of orders 
in the limit order book. 
\end{en-text}
%

\section{Quasi likelihood }
We shall consider estimation for the unknown parameter $\theta$. 
%
\begin{en-text}
Let $\calx$ denote the set of 
$\sfd$-dimensional increasing function 
$N=(N^\alpha)_{\alpha\in\cali}=(N^\alpha_t)_{t\in\hat{I};\alpha\in\cali}$ 
such that each component $N^\alpha$ starts from $0$ at $t={\colorr\hat{T}_0}$ 
and is piecewise constant with finite number of jumps 
of size $1$. Each element $N$ of $\calx$ is identified with 
the random 
finite sums of delta measures on $I$. 
Equip $\calx$ with the $\sigma$-field generated by the mappings 
$\calx\ni N\mapsto N(B)$ for Borel measurable sets $B$ in $I$. 
\end{en-text}
Suppose that we have observations 
\beas 
(N^{n,\alpha}_t)_{t\in I,\alpha\in\cali}, \quad
(X^{n,\beta}_t)_{t\in \hat{I},\beta\in\cali_0}, \quad
(g^{n,\alpha}(t,\theta))_{t\in{\colorr I},\alpha\in\cali,\theta\in\Theta}, \quad
(K^{n,\alpha}_\beta(t,s,\theta))_{t\in I,s\in{\colorr[\hat{T}_0,t)},\alpha\in\cali,\beta\in\cali_0,\theta\in\Theta}
\eeas
{\blue up to $\theta$.}
This is the case, for example, when $g^{n,\alpha}(t,\theta)$ is a function of $\theta$ and some  
observable covariate process, {\blue that is, 
$g^{n,\alpha}(t,\theta)={\mathfrak g}(Z_t,\theta)$ for some observable 
covariate process $Z_t$ 
and a given function ${\mathfrak g}$. }
\begin{en-text}
{\colorb Let $\tilde{I}=[T_0-\delta_1-R,T_1+\delta_1]$.}
Define 
$\lambda^n(t,\theta,\gamma,\bbN)
=(\lambda^{n,\alpha}(t,\theta,\gamma,\bbN))_{\alpha\in\cali}$ ($t\in{\colorb \tilde{I}}$) 
by 
\bea\label{251121-5} 
\lambda^{n,\alpha}(t,\theta,\gamma,\bbN)
&=& 
g^{n,\alpha}(t,\gamma)+\int_{{\colorr\hat{T}_0}}^{(t+\Delta_\beta-\Delta_\alpha)-}
K^{n,\alpha}_\beta({\colorr t,}t+\Delta_\beta-\Delta_\alpha-s,\gamma)n^{-1}d\bbN^\beta_s
\eea
For {\colorb $t\geq T_0-\delta_1-R$}, if $K^{n,\alpha}_\beta(t',t+\Delta_\beta-\Delta_\alpha{\colorb-s})\not=0$, then 
\beas
s\geq t+\Delta_\beta-\Delta_\alpha-R\geq {\colorb T_0-2(\delta_1+R)}\geq \hat{T}_0. 
\eeas
In particular, (\ref{251121-5}) can be written as 
\begin{screen}\vspace{-4mm}
\bea\label{251121-6} 
\lambda^{n,\alpha}(t,\theta,\gamma,\bbN)
&=& 
g^{n,\alpha}(t,\gamma)+\int_{{\colorr t+\Delta_\beta-\Delta_\alpha-R}}^{(t+\Delta_\beta-\Delta_\alpha)-}
K^{n,\alpha}_\beta({\colorr t,}t+\Delta_\beta-\Delta_\alpha-s,\gamma)n^{-1}d\bbN^\beta_s
\eea
\end{screen}
for $t\in {\colorb \tilde{I}}$. 
{\colorr
Formula (\ref{251121-6}) has advantage because the integrand is continuous on the interval. 
$R$ is assumed to be known as a part of knowledge of $K^{n,\alpha}_\beta$. 
If $K^{n,\alpha}_\beta$ vanishes smoothly on the positive axis, then 
the boundary effects by $R$ will not appear since we can use 
the expression (\ref{251121-5}) and smooth kernel $K^{n,\alpha}_\beta$ in this situation. }
\end{en-text}
A statistician models the phenomena by 
the point processes $N^\alpha$ 
with intensity processes $n\lambda^{n,\alpha}(t,\theta)$, choosing a large number $n$.  
Only the estimated function $n\lambda^{n,\alpha}$, not $\lambda^{n,\alpha}$, makes sense as a result of 
statistical analysis. 

We adopt the quasi log likelihood 
\beas 
l_n(\theta) 
&=& 
\sum_{\alpha\in\cali} \bigg(
\int_{T_0}^{T_1} \log [n\lambda^{n,\alpha}(t,\theta)]dN^{n,\alpha}_t
-\int_{T_0}^{T_1} [n\lambda^{n,\alpha}(t,\theta)-1]dt
\bigg)
\eeas
for observed point process $N^n$. 
Obviously, ``$-1$'' in the second integral can be eliminated for maximization. 
The factor ``$n$'' in the first integral is also unnecessary. 
Thus we can use 
\bea\label{251010-1}
\ell_n(\theta) 
&=& 
\sum_{\alpha\in\cali} \bigg(
\int_{T_0}^{T_1} \log\lambda^{n,\alpha}(t,\theta)dN^{n,\alpha}_t
-\int_{T_0}^{T_1} n\lambda^{n,\alpha}(t,\theta)dt
\bigg)
\eea
instead of $l_n(\theta)$. 
To estimate $\theta$, 
we will consider the quasi maximum likelihood estimator (QMLE), that is, a sequence of estimators $\hat{\theta}_n$ 
that maximizes or asymptotically maximizes $\ell_n(\theta)$. 
The quasi Bayesian estimator (QBE) is another option, as discussed later. 

{\colorb 
Hereafter, 
we suppose that $\Theta$ is a bounded open set in $\bbR^\sfp$ and satisfies 
\bea\label{260810-1}
\inf_{\theta\in\Theta} \mbox{Leb}\big(\big\{\theta'\in\Theta; \> |\theta'-\theta|<\ep\big\}\big)
&\geq& a_0\big(\ep^\sfp\wedge1)
\sskip(\ep>0)
\eea
for some positive constant $a_0$, 
where 
$\mbox{Leb}$ is the Lebesgue measure.
}
The true value of $\theta$ will be denoted by $\theta^*$.

\begin{en-text}
$\lambda^\infty(t,\theta^*,\gamma^*)=(\lambda^{\infty,\alpha}(t,\gamma^*))_\alpha$ is a unique solution of 
\beas 
\lambda^\infty(t,\theta^*,\gamma^*)
&=& 
g^\infty(t,\theta^*,\gamma^*)+\int_0^tK(t-s,\gamma^*)\lambda^\infty(s,\theta^*,\gamma^*)ds
\eeas
\end{en-text}

\begin{en-text}
Denote by $\bbP^n$ the distribution of the lagged multi-variate point process 
$\bbN$ given by (\ref{251117-2}), that is, 
$\bbP^n=(P^n)^{\bbN}$. 
\end{en-text}

\section{{\colorb Quasi maximum likelihood estimator by a classical approach}}
\subsection{{\colorb Consistency of the quasi maximum likelihood estimator}
}\label{260810-2}

\begin{en-text}
In what follows, we will implicitly suppose that all random fields appearing in the assumptions  
are separable. That is, when a supremum $\bar{Y}(\omega,x_1,...,x_{k-1})=\sup_{x_k\in I_k,...,x_k\in I_m}Y(\omega,x_1,...,x_m)$ of 
a random field $Y(\omega,x_1,...,x_m)$ appears, 
we always assume the measurability of $\bar{Y}$ in $(\omega,x_1,...,x_{k-1})$.  
\end{en-text}
{\colorb In Section \ref{260810-2}, we suppose that 
the function $\Theta\ni\theta\mapsto\lambda^n(t,\theta)$ has continuous extension to $\bar{\Theta}$.}  
%
Let $J=\{(t,s);\>s\in[\hat{T}_0,t),\>t\in I\}$. 
\def\mfi{\mathfrak I}
Denote by $\mfi$ the set of $\bbR_+$-valued left-continuous non-decreasing adapted processes on 
$\Omega\times I$. 
\bd
\item[[A1\!\!]]$_{\bar{j}}$
For each $n\in\bbN$, 
$K^n(t,s,\theta)$ is an $\bbR_+^\sfd\otimes\bbR_+^{\sfd_0}$-valued  
$\calf\times\bbB(J)\times\bbB(\Theta)$-measurable\footnote{
{\colorb $\bbB(A)$ is the Borel $\sigma$-field for a topological space $A$.}}
function satisfying the following conditions. 
\bd
\im[(i)] 
For each $(n,t,\theta)\in \bbN\times I\times\Theta$, the process 
$[\hat{T}_0,t)\ni s\mapsto K^n(t,s,\theta)$ is 
optional on $\Omega\times[\hat{T}_0,t)$, and 
each path is continuous or has jump discontinuity at every point $s$. 
\footnote{\blue A mapping $f$ on $[a,b)$ is continuous or has jump discontinuity at 
every point $s$ if $f(s+)$ exist for all $s\in[a,b)$ and 
if $f(s-)$ exist for all $s\in(a,b)$. } 

\im[(ii)] 
For each $(n,t,s)\in\bbN\times J$, the mapping 
$\Theta\ni\theta\mapsto K^{n}(t,s,\theta)$ is $\bar{j}$ times differentiable a.s., 
{\blue those derivatives are right or left-continuous in $s$ at every point $s\in[\hat{T}_0,t)$,}
and there exist $\dot{K}^n\in\mfi$ for $n\in\bbN$ such that 
\beas 
\sum_{j=0}^{\bar{j}}
\sup_{t'\in[{\blue T_0},t]}
\sup_{s\in[\hat{T}_0,t')}\sup_{\theta\in\Theta} |(\partial_\theta)^j K^n(t',s,\theta)|
&\leq& 
\dot{K}^n(t)\quad (t\in I,\>n\in\bbN)
\eeas
and that the family 
$\displaystyle\big\{\dot{K}^n(T_1)\big\}_{n\in\bbN}$ is tight. 
\begin{en-text}
\im[(iii)] {\colorr This condition can be reduced to only continuity in $s$?} 
For each $(n,t,\theta)\in\bbN\times I\times\Theta$, the mapping 
$[\hat{T}_0,t)\ni s\mapsto K^n(t,s,\theta)$ is differentialble a.s.,  
and there exist left-continuous adapted process $\dddot{K}^n$ on $\Omega\times I$ such that 
\beas 
\sup_{\theta\in\Theta} 
\int_{\hat{T}_0}^t| \partial_s K^n(t,s,\theta)|ds
&\leq& 
\dddot{K}^n(t)\quad (t\in I,\>n\in\bbN)
\eeas
is weakly locally bounded for each $n\in\bbN$ and that 
the family 
$\displaystyle\big\{\dddot{K}^n(T_1)\big\}_{n\in\bbN}$ is tight. 
\end{en-text}

\im[(iii)] 
There exist an $\bbR_+^\sfd\otimes\bbR_+^{\sfd_0}$-valued random field $K^\infty(t,s,\theta)$ 
that is differentiable in $\theta\in\Theta$, and 
there exist $\ddot{K}^n\in\mfi$ for $n\in\bbN$ such that 
\beas 
\sum_{j=0}^{1\wedge\bar{j}}
\sup_{t'\in[{\blue T_0},t]}\sup_{s\in[\hat{T}_0,t')}\sup_{\theta\in\Theta} 
\l|\partial_{\theta}^jK^{n}(t',s,\theta)
-\partial_{\theta}^jK^{\infty}(t',s,\theta)
\r|
&\leq& 
\ddot{K}^n(t)\quad (t\in I,\>n\in\bbN)
\eeas
and that $\ddot{K}^n(T_1)\to^p0$ 
as $n\to\infty$. 
 
\ed
\ed

\begin{description}
\item[[A2\!\!]]$_{\bar{j}}$ 
For each $(\alpha,n)\in\cali\times\bbN$, 
$g^{n,\alpha}(t,\theta)$ is an 
nonnegative 
$\calf\times\bbB(I)\times\bbB(\Theta)$-measurable function 
for which the following conditions are fulfilled. 
\bd
\im[(i)] For each $(n,\alpha,\theta)\in\bbN\times\cali\times\Theta$, 
the process $(g^{n,\alpha}(t,\theta))_{t\in I}$ is predictable.  

\begin{en-text}
Remove ! ============\\
the mapping 
\beas 
I\ni t\mapsto \inf_{\theta\in\Theta}g^{n,\alpha}(t,\theta)
\eeas
is $(0,\infty)$-valued and left-continuous a.s. for every $n\in\bbN$, and 
\beas 
\sum_{\alpha\in\cali}\sup_{t'\in[T_0,t]}
\sup_{\theta\in\Theta}g^{n,\alpha}(t',\theta)^{-1} &\leq& \odg^n(t)\quad(t\in I,\>n\in\bbN)
\eeas
for some $\odg^n\in\mfi$ for $n\in\bbN$ such that  
the family $\big\{\odg^n(T_1)\big\}_{n\in\bbN}$ 
is tight. \\
============
\end{en-text}

\im[(ii)] For each $(n,t)\in\bbN\times I$, the mapping $\Theta\ni\theta\mapsto g^{n}(t,\theta)$ is 
$\bar{j}$ times differentiable a.s., 
and there exist $\dodg^n\in\mfi$ for $n\in\bbN$ such that 
\beas 
\sum_{j=0}^{\bar{j}} 
\sup_{t'\in[T_0,t]}\sup_{\theta\in\Theta}\big|(\partial_\theta)^j
g^n(t',\theta)\big|
&\leq&
\dodg^n(t)\quad (t\in I,\>n\in\bbN)
\eeas
and that the family 
$\big\{\dodg^n(T_1)\big\}
_{n\in\bbN}$ 
is tight. 

\im[(iii)] 
There exist an $\bbR_+^\sfd$-valued random field $g^\infty(t,\theta)$ 
that is differentiable in $\theta\in\Theta$, and 
there exist $\ddodg^n\in\mfi$ for $n\in\bbN$ such that 
\beas 
\sum_{j=0}^{1\wedge\bar{j}}\sup_{t'\in[T_0,t]}\sup_{\theta\in\Theta}
\l| \partial_\theta^jg^{n}(t',\theta)
-\partial_\theta^jg^{\infty}(t',\theta)\r|
&\leq&
\ddodg^n(t)\quad(t\in I,\>n\in\bbN)
\eeas
and that $\ddodg^n(T_1)\to^p0$ 
as $n\to\infty$. 
\ed
\ed

\begin{remark*}\rm 
If one assumes separability for random fields on $(t,s,\theta)$ or $(t,\theta)$, 
it is possible to avoid introducing the envelope processes $\dot{K}^n$ etc. 
However, finding envelope processes is easier than verifying separability involving many intervals of variables. 
\end{remark*}
%

Let $\varrho:\calm_b(\hat{I})\times\calm_b(\hat{I})\to\bbR_+$ be a metric that is compatible with 
the weak$^*$-topology on the set of $\bbR$-valued measures with finite total variation. 
In other words, for $\mu_n,\mu\in\calm_b(\hat{I})$, 
the convergence 
$\varrho(\mu_n,\mu)\to0$ is equivalent to that $\mu_n(f)\to\mu(f)$ for all $f\in C(\hat{I})$. 
Nondecreasing functions will be identified with measures.

\bd
\item[[A3\!\!]]$_{{\colorrz \bar{j}}}$ 
For each $n\in\bbN$ and $\beta\in\cali_0$, $(X^{n,\beta}_t)_{t\in\hat{I}}$ is a non-decreasing right-continuous 
$(\calf_t)_{t\in\hat{I}}$-adapted process, and 
for each $\beta\in\cali_0$, there exists a 
non-decreasing process $(X^{\infty,\beta}_t)_{t\in\hat{I}}$ such that 
\begin{en-text}
\beas 
\sup_{t\in \hat{I}}\big|X^{n,\beta}_t
-X^{\infty,\beta}_t
\big|
&\to^p& 
0
\eeas
\end{en-text}
\beas 
\varrho(X^{n,\beta},X^{\infty,\beta}) &\to^p& 
0
\eeas
as $n\to\infty$ and that 
\beas \int_{[\hat{T}_0,t)}
{\blue 1_{\cald_\alpha^\beta(t,\theta)}(s)}
dX^{\infty,\beta}_s &=& 0
\eeas
for all $(t,\theta,\alpha,\beta)\in I\times\Theta\times\cali\times\cali_0$ a.s.\footnote{That is, 
this equality holds for all $(t,\theta)\in I\times\Theta$ on some event $\Omega_0\in\calf$ with $P[\Omega_0]=1$. 
}
{\blue where $\cald_\alpha^\beta(t,\theta)=\{s\in[\hat{T}_0,t);\>\partial_\theta^jK(t,s,\theta)^{\infty,\alpha}_\beta$ is discontinuous at $s$ for some $j\leq\bar{j}\}$. }

\ed
\halflineskip

We do not assume continuity of $K^\infty$ in $s$, 
{\blue which is necessary to treat a kernel looking back a finite-length of history. }

\begin{lemma*}\label{260808-3}
Let $j\in\{0,1\}$. Then 
under $[A1]_{j}$ $(ii)$ and $[A3]{\colorrz _{j}}$,
\bea\label{260808-1}
\int_{I\times\Theta} \mu_1(dt)\mu_2(d\theta)\bigg|\int_{[\hat{T}_0,t)}
\partial_\theta^jK^\infty(t,s,\theta)dX^n_s
-
\int_{[\hat{T}_0,t)}\partial_\theta^jK^\infty(t,s,\theta)dX^\infty_s\bigg|
&\to^p&0
\eea
as $n\to\infty$ for any a.s.-bounded continuous random measure $\mu_1$ on $I$ and 
any a.s.-bounded random measure $\mu_2$ on $\Theta$. 
Additionally under $[A1]_{{\colorrz j}}$ $(iii)$, 
\bea\label{260808-2}
\int_{I\times\Theta} \mu_1(dt)\mu_2(d\theta)\bigg|\int_{[\hat{T}_0,t)}\partial_\theta^jK^n(t,s,\theta)dX^n_s
-
\int_{[\hat{T}_0,t)}\partial_\theta^jK^\infty(t,s,\theta)dX^\infty_s\bigg|
&\to^p&0\eea
as $n\to\infty$. 
%
\end{lemma*}
\proof Let $j=0$. 
To prove (\ref{260808-1}) with the subsequence argument, 
we may assume the convergence in [A3]$\colorrz_0$ holds a.s. and 
we will consider $\omega$ for which this convergence occurs 
as well as 
$\mu_1(\omega,\cdot)$ and $\mu_2(\omega,\cdot)$ are bounded. 
Moreover, we may assume the boundedness of $K^\infty$ thanks to [A1]$_0$ (ii). 
Then there is a subset $D_\omega$ of $\hat{I}$ such that 
$\hat{I}\setminus D_\omega$ is at most countable and that 
\beas\varrho\big(X^{n,\beta}|_{[\hat{T}_0,t)}, X^{\infty,\beta}|_{[\hat{T}_0,t)}\big)\to0\eeas 
for every $t\in D_\omega$. 
Due to the second condition of [A3]$_{\colorrz 0}$, 
\beas 
\int_{[\hat{T}_0,t)}K^\infty(t,s,\theta)dX^n_s &\to& 
\int_{[\hat{T}_0,t)}K^\infty(t,s,\theta)dX^\infty_s
\eeas
as $n\to\infty$ for every $(t,\theta)\in D_\omega\times\Theta$. 
Convergence of [A3]$_{\colorrz 0}$ also implies the boundedness of $\{X^{n,\beta}(\hat{I})\}_{n\in\bbN}$. 
Then the dominated convergence theorem gives 
\beas
\int_{I\times\Theta} \mu_1(dt)\mu_2(d\theta)
{\blue \bigg|}
\int_{[\hat{T}_0,t)}K^\infty(t,s,\theta)dX^n_s
{\blue -}
\int_{[\hat{T}_0,t)}K^\infty(t,s,\theta)dX^\infty_s
{\blue \bigg|}&\to&{\blue 0}
\eeas
for the $\omega$. 
This proves the convergence (\ref{260808-1}), and as a result (\ref{260808-2}). 
The convergences in the case $j=1$ are verified in the same way. 
\qed\halflineskip

{\colorr 
\begin{description}
\item[[A4\!\!]] 
For each $(\omega,n,\alpha,t,\theta)\in\Omega\times\bbN\times\cali\times I\times\Theta$, 
$\lambda^{n,\alpha}(t,\theta)=0$ if and only if $\lambda^{n,\alpha}(t,\theta^*)=0$, and 
\beas 
\sum_{\alpha\in\cali}\sup_{t'\in[T_0,t]}
\sup_{\theta\in\Theta}\big\{\lambda^{n,\alpha}(t',\theta)^{-1} 1_{\{\lambda^{n,\alpha}(t',\theta)\not=0\}}\big\}
&\leq& \odg^n(t)\quad(t\in I,\>n\in\bbN)
\eeas
for some $\odg^n\in\mfi$ for $n\in\bbN$ such that  
the family $\big\{\odg^n(T_1)\big\}_{n\in\bbN}$ 
is tight. 
\end{description}
}
%
\begin{remark}\rm {\blue 
For modeling of $C^\alpha$ and $C^\beta$ 
of the limit order book in Section \ref{270723-1}, 
we may {\blue consider} $g^{n,\alpha}(t,\theta)$ proportional 
to $A^\alpha$ or $B^\beta$, or more complicated mechanism. 
Non degeneracy of the intensity processes 
for validating likelihood analysis 
seems to be problematic due to the shape of the quasi likelihood, 
but it causes no difficulty thanks to a positive minimum unit of orders 
in the limit order book. }
\end{remark}

Let 
\bea\label{260429-111}
\tilde{N}^{n,\alpha}_t=N^{n,\alpha}_t-N^{n,\alpha}_{T_0}-\int_{T_0}^tn\lambda^{n,\alpha}(s,\theta^*)ds
\quad (t\in I)
\eea
\begin{lemma*}\label{251123-1} 
Suppose that Conditions $[A1]_0$ $(i), (ii)$, $[A2]_0$ $(i), (ii)$, and $[A3]_{\colorrz 0}$ are fulfilled. 
\bd\im[(a)] 
The family $\big\{
\big|\lambda^n(t,\theta^*)\big|\big\}_{(n,t)\in\bbN\times I}$ is tight. 
\im[(b)] 
The family $\bigl\{n^{-1}N^{n,\alpha}([T_0,T_1])\bigr\}_{\alpha\in\cali,n\in\bbN}$ is tight. 
\im[(c)] 
The process $(\tilde{N}^{n,\alpha}_t)_{t\in I}$ 
is a locally square-integrable martingale with 
\beas 
\langle \tilde{N}^{n,\alpha}\rangle_t &=& n\int_{T_0}^t \lambda^{n,\alpha}(s,\theta^*)ds.
\eeas
\im[(d)] 
The family 
$\l\{\displaystyle 
\sup_{t\in I}\l|n^{-1/2}\tilde{N}^{n,\alpha}_t\r|\r\}_{n\in\bbN}
$ is tight. 
\ed
\end{lemma*}
\proof  By positivity of processes, 
\beas 
\lambda^{n,\alpha}(t,\theta) 
&\leq& 
g^{n,\alpha}(t,\theta)+ \sum_{\beta\in\cali_0}
\sup _{(t,s)\in J} K^{n,\alpha}_\beta(t,s,\theta)\big(X^{n,\beta}_{T_1}-X^{n,\beta}_{\hat{T}_0}\big)
\\&\leq&
\dodg^n(T_1)+\dot{K}^n(T_1)\big(X^{n,\beta}_{T_1}-X^{n,\beta}_{\hat{T}_0}\big)
\eeas
for all $t\in I$ and $\theta\in\Theta$. 
\begin{en-text}
We note that the weak local boundedness assumptions 
in [A3]$_{\colorrz 0}$, [A1]$_0$ (ii), and [A2]$_0$ (ii) 
ensure that of 
$I\ni t\mapsto\sup_{\theta\in\Theta}\lambda^{n,\alpha}(t,\theta)$ for 
every $(n,\alpha)\in\bbN\times\cali$. 
Moreover the tightness assumptions imply that of 
$\big\{\sup_{t\in I}\sup_{\theta\in\Theta}\lambda^{n,\alpha}(t,\theta)\big\}_{n\in\bbN}$. 
\end{en-text}
Therefore (a) follows. 
In particular, 
$\int_{T_0}^{T_1}\lambda^{n,\alpha}(t,\theta^*)dt<\infty$ a.s., therefore (\ref{260429-111}) is well defined. 
Now it is easy to see (a) $\Iku$ (b) $\Iku$ (c) $\Iku$ (d). 
\qed
\halflineskip

For $\ell_n$ in (\ref{251010-1}), let 
\beas 
\bbY_n(\theta)
&=&
\frac{1}{n}[\ell_n(\theta)-\ell_n(\theta^*)]
\\&\equiv&
\sum_{\alpha=1}^\sfd\bigg(
\int_{T_0}^{T_1} \log\frac{\lambda^{n,\alpha}(t,\theta)}
{\lambda^{n,\alpha}(t,\theta^*)}
n^{-1}{\blue d}N^{n,\alpha}_t
-\int_{T_0}^{T_1}
\big[\lambda^{n,\alpha}(t,\theta)-\lambda^{n,\alpha}(t,\theta^*)\big]dt
\bigg).
\eeas
Here it should be noted that 
the random fields $\ell_n$ and $\bbY_n$ are well defined 
{\colorr thanks to [A4]. }
\begin{en-text}
In Section \ref{260810-2}, we suppose that a bounded open set $\Theta$ in $\bbR^\sfp$ satisfies 
the condition that 
\bea\label{260810-1}
\inf_{\theta\in\Theta} \mbox{Leb}\big(\big\{\theta'\in\Theta; \> |\theta'-\theta|<\ep\big\}\big)
&\geq& a_0\big(\ep^\sfp\wedge1)
\sskip(\ep>0)
\eea
for some positive constant $a_0$, 
where 
$\mbox{Leb}$ is the Lebesgue measure. 
Moreover we suppose that 
the function $\Theta\ni\theta\mapsto\lambda^n(t,\theta)$ has continuous extension to $\bar{\Theta}$. 
\halflineskip
\end{en-text}
Let 
\bea\label{251124-201}
{\lambda}^{\infty,\alpha}(t,\theta)
&=&
g^{\infty,\alpha}(t,\theta)
+
{\blue \sum_{\beta\in\cali_0}}
\int^{t-}_{\hat{T}_0}
K^{\infty,\alpha}_{\beta}(t,s,\theta)
dX^{\infty,\beta}_s
\eea
for $t\in I$ {\blue and $\theta\in\Theta$}. 
\begin{en-text}
\koko
By Sobolev's inequality, 
\beas 
\sup_{\theta\in\Theta} \big|\lambda^n(t,\theta)-\lambda^\infty(t,\theta)\big|
&\leq&
C_k(\Theta)
\bigg\{\sum_{j=0}^1 \int_\Theta  
\big|\partial_\theta^j\lambda^n(t,\theta)-\partial_\theta^j\lambda^\infty(t,\theta)\big|^{k}d\theta
\bigg\}^{1/k}
\\&\leq&
C_k(\Theta)\bigg[|\Theta|^{1/k}\ddodg^n(T_1)
\eeas
where $k>\sfp$ and 
$C_k(\Theta)$ is a constant only depending on $\Theta$ and $k$.

\begin{lemma*}\label{260604-1}
{\coloro Uniformity in $t$ is not necessary. }
Suppose that Conditions $[A1]_1$, $[A2]_1$ and $[A3]_{\colorrz 0}$ hold. Then 
\beas 
\sup_{t\in I,\theta\in\Theta} \big|\lambda^{n,\alpha}(t,\theta)-\lambda^{\infty,\alpha}(t,\theta)\big|
&\to^p& 
0
\eeas
as $n\to\infty$. 
\end{lemma*}
\proof 
We have 
\beas 
\lambda^n(t,\theta) -\lambda^\infty(t,\theta)
&=& 
g^n(t,\theta)-g^\infty(t,\theta)
+
\int_{\hat{T}_0}^{t-} K^n(t,s,\theta)dX^n_s - \int_{\hat{T}_0}^{t-} K^\infty(t,s,\theta)dX^\infty_s 
\\&=& 
g^n(t,\theta)-g^\infty(t,\theta)
+
\int_{\hat{T}_0}^{t-} (K^n(t,s,\theta)-K^\infty(t,s,\theta))dX^\infty_s 
\\&&
+ \int_{\hat{T}_0}^{t-} K^n(t,s,\theta)dX^n_s 
- \int_{\hat{T}_0}^{t-} K^n(t,s,\theta)dX^\infty_s 
\\&=&
g^n(t,\theta)-g^\infty(t,\theta)
+
\int_{\hat{T}_0}^{t-} (K^n(t,s,\theta)-K^\infty(t,s,\theta))dX^\infty_s 
\\&&
+K^n(t,t-,\theta)\big(X^n_{t-}-X^\infty_{t-}\big)
-K^n(t,\hat{T}_0,\theta)\big(X^n_{\hat{T}_0}-X^\infty_{\hat{T}_0}\big)
\\&&
-\int_{\hat{T}_0}^{t-} (\partial_sK^n)(t,s,\theta)\big(X^n_s- X^\infty_s\big)ds. \qed
\eeas
\halflineskip
\end{en-text}
%
%

\begin{lemma*}\label{260803-10}
Suppose that $[A1]_1$, $[A2]_1$, $[A3]_{\colorrz 0}$ {\colorr and $[A4]$} are fulfilled. 
Then the family 
\beas 
\bigg\{\sup_{\theta\in \Theta} \l|\int_{T_0}^{T_1}
\log\frac{{\lambda}^{n,\alpha}(t,\theta)}
{{\lambda}^{n,\alpha}(t,\theta^*)}n^{-1/2}d\tilde{N}^{n,\alpha}_t
\r|\bigg\}_{n\in\bbN}
\eeas
is tight. 
\end{lemma*}
\proof 
Before starting the proof, we note that 
the stochastic integrals of the statement are continuous in $\theta$ 
under the assumptions, and hence the supremum is measurable. 

Let 
\beas 
M^n_t(\theta) 
&=& 
\int_{T_0}^t
\xi^{n,\alpha}(s,\theta)n^{-1/2}d\tilde{N}^{n,\alpha}_s
\quad(t\in I)
\eeas
with 
\beas 
\xi^{n,\alpha}(s,\theta) &=& 
{\colorr 1_{\{{\lambda}^{n,\alpha}(s,\theta^*)\not=0\}}}
\log\frac{{\lambda}^{n,\alpha}(s,\theta)}
{{\lambda}^{n,\alpha}(s,\theta^*)}
\eeas
%
%
Let $A>0$. 
Define stopping times $\tau^n_A$ depending on $A>0$ by 
\bea\label{260809-7} 
\tau^n_A &=& 
\inf\big\{t\geq T_0;\> \dot{K}^n(t)+\ddot{K}^n(t)
+\odg^n(t)+\dodg^n(t)+\ddodg^n(t)
{\blue +|X^n_{t-}-X^n_{T_0}|}
\> >\> A\big\}\wedge T_1.
\eea
Then for an event $\Omega_0\in\calf$ with $P(\Omega_0)=1$ and a non-decreasing function ${\frak f}:\bbR_+\to\bbR_+$, 
it holds that 
\beas 
\sup_{\omega\in\Omega_0}\sup_{n\in\bbN}
\sup_{t\in[\hat{T}_0,\tau^n_A]}\sup_{\theta\in\Theta}
\big(1_{\{\tau^n_A>0\}}\lambda^{n,\alpha}(t,\theta)+1_{\{\tau^n_A>0\}}\big|\xi^{n,\alpha}(t,\theta)\big|\big)
&\leq&
{\frak f}(A).
\eeas
Here the left-continuity of the dominating functions worked. 

By the Burkholder-Davis-Gundy inequality, for $k>1/2$, 
\bea\label{260809-1} &&
E\Biggl[\Biggl|
\int_{T_0}^{t\wedge\tau^n_A}
(\partial_\theta)^j\xi^{n,\alpha}(s,\theta)n^{-1/2}d\tilde{N}^{n,\alpha}_s
\Bigg|^{2k}\Bigg]
\nn\\&\leq&
C_k
E\Biggl[\Biggl|
n^{-1}\int_{T_0}^{t\wedge\tau^n_A}
\big((\partial_\theta)^j\xi^{n,\alpha}(s,\theta)\big)^2dN^{n,\alpha}_s
\Bigg|^{k}\Bigg]
\nn\\&\leq&
2^{k-1}C_k
n^{-k/2}
E\Biggl[\Biggl|
\int_{T_0}^{t\wedge\tau^n_A}
\big((\partial_\theta)^j\xi^{n,\alpha}(s,\theta)\big)^2n^{-1/2}d\tilde{N}^{n,\alpha}_s
\Bigg|^{k}\Bigg]
\nn\\&&
+2^{k-1}C_k
E\Biggl[\Biggl|
\int_{T_0}^{t\wedge\tau^n_A}
\big((\partial_\theta)^j\xi^{n,\alpha}(s,\theta)\big)^2\lambda^{n,\alpha}(s,\theta^*)ds
\Bigg|^{k}\Bigg].
\eea
Repeatedly using (\ref{260809-1}) starting with $k=2^{\blue m}$ for ${\blue m}\in\bbN$ and 
$E[(\int \cdots n^{-1/2}d\tilde{N}^{n,\alpha})^2]=E[\int \cdots^2\lambda^{n,\alpha}(s,\theta^*)ds]$ at last, we obtain 
\bea\label{260809-2} 
\sup_{(n,t,\theta)\in\bbN\times I\times\Theta}E\Biggl[\Biggl|
\int_{T_0}^{t\wedge\tau^n_A}
(\partial_\theta)^j\xi^{n,\alpha}(s,\theta)n^{-1/2}d\tilde{N}^{n,\alpha}_s
\Bigg|^{p}\Bigg]
&<&
\infty
\eea
for every $p>1$ and $j=0,1$. 

Applying Sobolev's inequality to ${\blue \Theta}$ and using (\ref{260809-2}), we have 
\bea\label{260809-3}
\sup_{(n,t)\in\bbN\times I}E\bigg[\sup_{\theta\in\Theta}\big|M^n_{t\wedge\tau_n}(\theta)\big|^p\bigg]
&\leq&
\sup_{(n,t)\in\bbN\times I}C_p(\Theta)
E\Bigl[\sum_{j=0}^1 \int_{\blue \Theta}\Bigl((\partial_\theta)^j
M^n_{t\wedge\tau_n}(\theta)
\Big)^p d\theta \Big]
\nn\\&=&
\sup_{(n,t)\in\bbN\times I}C_p(\Theta)
\sum_{j=0}^1 \int_{\blue \Theta} d\theta 
E\Bigl[\Bigl((\partial_\theta)^j
M^n_{t\wedge\tau_n}(\theta)
\Big)^p\Big]
\nn\\&<&
\infty
\eea
for $p>\sfp$. 

It follows from the tightness of $\{\dot{K}^n(T_1),...,\ddodg^n(T_1),
{\blue |X^n_{T_1}-X^n_{T_0}|}\}_{n\in\bbN}$ that 
for any $\ep>0$, there exists $A>0$ such that 
$\sup_{n\in\bbN} P[\tau^n_A<T_1]<\ep$. 
This with Inequality (\ref{260809-3}) proves the result.  
\qed\halflineskip

\begin{en-text}
\koko 
In view of Lemma \ref{251123-1} (a), (c) , Condition [A2]$_1$ (i) 
and [A2]$_1$ (ii), we see that 
for $\ep>0$, there exists a number $A>0$ such that 
\beas 
\sup_{n\in\bbN}P\bigg[n^{-1}\langle \tilde{N}^{n,\alpha}\rangle_{T_1}\geq A\bigg] &<& \frac{\ep}{4},
\eeas
\beas 
\sup_{n\in\bbN}P\bigg[\min_{\alpha\in\cali}\inf_{(t,\theta)\in I\times\Theta}g^{n,\alpha}(t,\theta)
\leq \frac{1}{A}\bigg] &<& \frac{\ep}{4}
\eeas

\beas 
\sup_{n\in\bbN}P\bigg[
\sum_{j=0}^{\bar{j}}\sup_{t\in I}\sup_{\theta\in\Theta}\big|(\partial_\theta)^jg^n(t,\theta)\big|
\geq A\bigg] &<& \frac{\ep}{4}
\eeas
{\coloro $\up$ To be corrected. The sets inside of $P$ is not necessarily measurable. }
\end{en-text}

Now 
\beas 
\bbY_n(\theta) 
&=& 
\sum_{\alpha=1}^\sfd\bigg(
\int_{T_0}^{T_1} \log\frac{{\lambda}^{n,\alpha}(t,\theta)}
{{\lambda}^{n,\alpha}(t,\theta^*)}
n^{-1}d\tilde{N}^{n,\alpha}_t\\
&&
-\int_{T_0}^{T_1}
\bigg[{\lambda}^{n,\alpha}(t,\theta)-{\lambda}^{n,\alpha}(t,\theta^*)
- \log\frac{{\lambda}^{n,\alpha}(t,\theta)}
{{\lambda}^{n,\alpha}(t,\theta^*)}
\lambda^{n,\alpha}(t,\theta^*)
\bigg]dt
\bigg).
\eeas
Let 
\bea\label{260810-11}
\bbY(\theta)
&=&
-\sum_{\alpha=1}^\sfd\int_{T_0}^{T_1}
\bigg[{\lambda}^{\infty,\alpha}(t,\theta)-{\lambda}^{\infty,\alpha}(t,\theta^*)
- \log\frac{{\lambda}^{\infty,\alpha}(t,\theta)}
{{\lambda}^{\infty,\alpha}(t,\theta^*)}
{\lambda}^{\infty,\alpha}(t,\theta^*)\bigg]dt
\eea

\begin{lemma*}\label{251125-1} 
{\blue Under $[A1]_1$, $[A2]_1$, $[A3]_0$ and $[A4]$, 
Then $\bbY$ has a continuous extension to $\bar{\Theta}$ and 
}
\beas
\sup_{\theta\in{\colorb \bar{\Theta}}}\big|\bbY_n(\theta)-\bbY(\theta)\big|
&\to^p& 0.
\eeas
\end{lemma*}
\proof We shall use the stopping times $\tau^n_A$ given in (\ref{260809-7}) for $A>0$. 
Let 
\beas
\bar{\bbY}_n(\theta)
&=&
-\sum_{\alpha=1}^\sfd\int_{T_0}^{T_1}
\bigg[{\lambda}^{n,\alpha}(t,\theta)-{\lambda}^{n,\alpha}(t,\theta^*)
- \log\frac{{\lambda}^{n,\alpha}(t,\theta)}
{{\lambda}^{n,\alpha}(t,\theta^*)}
\lambda^{n,\alpha}(t,\theta^*)
\bigg]dt.
\eeas
Then, on the event $\{\tau^n_A=T_1\}$, 
\beas &&
|\bar{\bbY}_n(\theta) - \bbY(\theta) |
\\&\leq& 
C_A\bigg(1+|X^n_{T_1}-X^n_{\hat{T}_0}|+|X^\infty_{T_1}-X^\infty_{\hat{T}_0}|\bigg)
\\&&
\times
\sum_{\alpha=1}^\sfd\bigg\{
\int_{T_0}^{T_1}\big| \lambda^{n,\alpha}(t,\theta)- \lambda^{\infty,\alpha}(t,\theta)\big|dt
+\int_{T_0}^{T_1}\big| \lambda^{n,\alpha}(t,\theta^*)- \lambda^{\infty,\alpha}(t,\theta^*)\big|dt\bigg\},
\eeas
where $C_A$ is a constant depending on $A$. 
Therefore, due to Lemma \ref{260808-3}, we have 
\bea\label{260809-10}
\bar{\bbY}_n(\theta) &\to^p& \bbY(\theta)
\eea
as $n\to\infty$ for each $\theta\in\Theta$. 
Since 
\beas 
\sup_{\theta\in\Theta}\big|\partial_\theta\bar{\bbY}_n(\theta) \big|
&\leq& 
\sfd |I| \big(1+\odg^n(T_1)\big)\big\{1+\dodg^n(T_1)+\dot{K}^n(T_1)(X^n_{T_1}-X^n_{\hat{T}_0}\big)\big\}^2
\eeas
and the family of random variables on the right-hand is tight, 
the family $\{P^{\bar{\bbY}_n}\}_{n\in\bbN}$ is tight as the family of distributions on $C(\bar{\Theta})$ 
due to the existence of continuous extension of $\bar{\bbY}_n$ to $\bar{\Theta}$. 
In particular, 
{\colorb $\bbY$ is well defined as a continuous function on $\bar{\Theta}$ and}
(\ref{260809-10}) holds for all $\theta\in\bar{\Theta}$. 
Thus we can conclude according to the standard argument, 
\beas 
\sup_{\theta\in\bar{\Theta}}\big|\bar{\bbY}_n(\theta)-\bbY(\theta)\big|
&\to^p&
0
\eeas
as $n\to\infty$, which gives the result if combined with Lemma \ref{260803-10}. 
\qed\halflineskip

We assume 
\bd
\item[[{\colorr A5}\!\!]] 
For every $\ep>0$, 
$\displaystyle
\inf_{\theta\in\Theta: \atop
|(\theta)-(\theta^*)|>\ep}
\bbY(\theta)
<
0\sskip a.s.
$
\ed

\begin{en-text}
{A sequence of estimators $\hat{\theta}_n$ for $\theta$ is referred to as 
an approximate maximum likelihood estimator if 
it satisfies 
\beas 
\frac{1}{n}\ell_n(\hat{\theta}_n,\hat{\gamma}_n,\bbN) &\geq& 
\frac{1}{n}\ell_n(\theta^*,\gamma^*,\bbN)-o_p(1)
\eeas
as $n\to\infty$. 
\end{en-text}

The following theorem gives consistency of the approximate maximum likelihood estimator. 
\begin{theorem*}\label{270615-1}
Suppose that $[A1]_1$, $[A2]_1$, $[A3]_{\colorrz 0}$, {\colorr $[A4]$ and $[A5]$} are satisfied. Then 
any estimator $\hat{\theta}_n$ for $\theta$ 
{\blue satisfying} $n^{-1}\ell_n(\hat{\theta}_n)\geq n^{-1}\ell_n(\theta^*)-o_p(1)$ as $n\to\infty$ 
is consistent, that is, 
$
\hat{\theta}_n \to^p \theta^*
$ 
as $n\to\infty$. 
\end{theorem*}

\begin{remark*}\rm 
Theorems \ref{270615-1} and  \ref{260716-4} are regarded as generalizations of 
Theorems 1 and 2 of Chen and Hall \cite{chen2013inference}, respectively. 
The consistency result given here is asserted for any sequence of 
quasi maximum likelihood estimator. 
The limit theorem gives asymptotic mixed normality in a general regression scheme 
in non-ergodic statistics. 
Though the treatments are simpler in the classical  methods, 
they are not sufficient to develop advanced themes 
such as prediction, information criteria and higher-order asymptotic theory. 
In particular, convergence of the moments of the quasi likelihood estimators or equivalently 
sharp estimates of tail probability of them is indispensable. 
The quasi likelihood analysis with the polynomial type large deviation inequalities 
for statistical random fields 
will be established in Section \ref{qla-pp}, the main part of this article. 
This construction of inferential theory enables us to approach the above mentioned problems. 
Indeed, the reader can find in 
\cite{UchidaYoshida2013} and \cite{UchidaYoshida2015} such a flow from  
non-ergodic statistical inference for volatility to an information criterion for 
volatility model selection. 
When the input intensities of the Hawkes type processes are time-varying, 
the question of non-degeneracy of the statistical model becomes complicated than expected. 
We will give a sufficient condition for non-degeneracy. 
Asymptotic properties of the quasi Bayesian estimator will be elucidated as well. 
\end{remark*}
\halflineskip

\subsection{Asymptotic mixed normality of the QMLE: a classical approach}\label{260730-1}

We shall investigate the asymptotic distribution of the QMLE. 
In Section \ref{260730-1}, the parameter space $\Theta$ is assumed only to be open 
without Condition (\ref{260810-1}) because only local properties are discussed. 
\begin{en-text}
For it, we will apply localization by 
\beas 
\sigma_\ep &=& \inf\big\{t;\>\inf_{n\in\bbN,\theta\in\Theta}g^{n,\alpha}(t,\theta)<\ep\big\}\wedge T_1
\eeas
for $\ep>0$, and we may assume hereafter that 
\beas 
\inf_{\omega\in\Omega, n\in\bbN,t\in I,\theta\in\Theta}\lambda^n(t,\theta) &>& 0,
\eeas
thanks to [A2]$_0$ (i). 
\end{en-text}

Under regularity conditions stated later, we have
\bea\label{260716-3} 
\partial_\theta\ell_n(\theta^*) &=& 
\sum_\alpha \int_{T_0}^{T_1} \lambda^{n,\alpha}(t,\theta^*)^{-1}\partial_\theta\lambda^{n,\alpha}(t,\theta^*)
d\tilde{N}^{n,\alpha}_t
\eea
and
\bea\label{260716-1}
n^{-1}\partial_\theta^2\ell_n(\theta) &=& 
\sum_\alpha \int_{T_0}^{T_1} 
\partial_\theta\big(\partial_\theta\lambda^{n,\alpha}/\lambda^{n,\alpha}\big)(t,\theta)n^{-1}d\tilde{N}^{n,\alpha}_t
\nn\\&&
-\sum_\alpha \int_{T_0}^{T_1} 
(\partial_\theta\lambda^{n,\alpha})^{\otimes2}(t,\theta)\big(\lambda^{n,\alpha}(t,\theta)\big)^{-2}
\lambda^{n,\alpha}(t,\theta^*)dt
\nn\\&&
{\blue -}\sum_\alpha \int_{T_0}^{T_1} 
\partial_\theta^2\lambda^{n,\alpha}(t,\theta)\big(\lambda^{n,\alpha}(t,\theta)\big)^{-1}
\big(\lambda^{n,\alpha}(t,\theta)-
\lambda^{n,\alpha}(t,\theta^*)\big)dt.
\eea

We obtain the following lemma 
in the same way as the proof of Lemma \ref{260803-10}, replacing 
$\xi^{n,\alpha}(s,\theta)$ by 
$\partial_\theta\big(\partial_\theta\lambda^{n,\alpha}/\lambda^{n,\alpha}\big)(s,\theta)$ in this case. 

\begin{lemma*}\label{260617-1}
Suppose that $[A1]_3$, $[A2]_3$, $[A3]_{\colorrz 1}$ {\colorr and $[A4]$} are fulfilled. 
Then for any ball $V$ centered at $\theta^*$ in $\Theta$, 
the family 
\beas 
\bigg\{
\sup_{\theta\in V} \l|\int_{T_0}^{T_1}
\partial_\theta\big(\partial_\theta\lambda^{n,\alpha}/\lambda^{n,\alpha}\big)(t,\theta)n^{-1/2}d\tilde{N}^{n,\alpha}_t\r|
\bigg\}_{n\in\bbN}
\eeas
is tight. 
\end{lemma*}
\begin{en-text}
\proof 
Localize martingales and apply the BDG inequality and Sobolev's inequality to $V$. 
That is, estimate 
\beas 
\sum_{j=0}^1 \int_V d\theta 
E\Bigl[\Bigl((\partial_\theta)^j\int_{T_0}^{\tau_\ep}
\partial_\theta\big(\partial_\theta\lambda^{n,\alpha}/\lambda^{n,\alpha}\big)(t,\theta)n^{-1}d\tilde{N}^{n,\alpha}_t
\Big)^{2k}\Big]
\eeas
for $k\in\bbN$ and $\ep\in(0,1)$ for which all ovally bounded objects are bounded in $[\ep,\ep^{-1}]$ 
under localization by a stopping time $\tau_\ep$. 
\end{en-text}
\halflineskip

Let 
\bea\label{270614-1}
\Gamma 
&=& 
\sum_{\alpha\in\cali} \int_{T_0}^{T_1} 
(\partial_\theta\lambda^{\infty,\alpha})^{\otimes2}(\lambda^{\infty,\alpha})^{-1}(t,\theta^*)dt.
\eea

\begin{lemma*}\label{260618-1} 
{\blue Suppose that $[A1]_2$, $[A2]_2$, $[A3]_1$ and $[A4]$ are fulfilled. } 
For any sequence $(V_n)_{n\in\bbN}$ of neighborhoods of $\theta^*$ that is shrinking to $\{\theta^*\}$, 
\beas 
\sup_{\theta\in V_n}\l|n^{-1}\partial_\theta^2\ell_n(\theta) +\Gamma\r| 
&\to^p& 0
\eeas
as $n\to\infty$. 
\end{lemma*}
\proof 
Since $g^\infty(t,\theta)$ and $K^\infty(t,s,\theta)$ are continuous in $\theta$, and 
bounded a.s., so is $\lambda^\infty(t,\theta)$, that is defined by (\ref{251124-201}). 
We use the stopping times $\tau^n_A$ given in (\ref{260809-7}) for $A>0$. 
On the event $\{\tau^n_A=T_1\}$, 
\bea\label{260716-1}&&
\big|n^{-1}\partial_\theta^2\ell_n(\theta) -\Gamma\big|
\nn\\&\leq& 
\sum_\alpha \bigg|\int_{T_0}^{T_1} 
\partial_\theta\big(\partial_\theta\lambda^{n,\alpha}/\lambda^{n,\alpha}\big)(t,\theta)n^{-1}d\tilde{N}^{n,\alpha}_t\bigg|
\nn\\&&
+\sum_\alpha \bigg|\int_{T_0}^{T_1} 
(\partial_\theta\lambda^{n,\alpha})^{\otimes2}(t,\theta)\big(\lambda^{n,\alpha}(t,\theta)\big)^{-2}
\lambda^{n,\alpha}(t,\theta^*)dt
\nn\\&&
\qquad-\int_{T_0}^{T_1} 
(\partial_\theta\lambda^{\infty,\alpha})^{\otimes2}(t,\theta^*)\big(\lambda^{\infty,\alpha}(t,\theta^*)\big)^{-2}
\lambda^{\infty,\alpha}(t,\theta^*)dt\bigg|
\nn\\&&
+\sum_\alpha \bigg|\int_{T_0}^{T_1} 
\partial_\theta^2\lambda^{n,\alpha}(t,\theta)\big(\lambda^{n,\alpha}(t,\theta)\big)^{-1}
\big(\lambda^{n,\alpha}(t,\theta)-
\lambda^{n,\alpha}(t,\theta^*)\big)dt\bigg|
\nn\\&\leq& 
\sum_\alpha \bigg|\int_{T_0}^{T_1} 
\partial_\theta\big(\partial_\theta\lambda^{n,\alpha}/\lambda^{n,\alpha}\big)(t,\theta)n^{-1}d\tilde{N}^{n,\alpha}_t\bigg|
\nn\\&&
\qquad+C\big(1+\odg^n(T_1)\big)^3
\big\{{\blue 1+}\dodg^n(T_1)+\dot{K}^n(T_1)|X^n(T_1)-X^n(\hat{T}_0)|\big\}^{\blue 3}
\nn\\&&
\times\bigg\{\sum_{j=0}^1\int_{T_0}^{T_1} \big|\partial_\theta^j\lambda^{n,\alpha}(t,\theta)
-\partial_\theta^j\lambda^{\infty,\alpha}(t,\theta)\big|dt
+\int_{T_0}^{T_1} \big|\lambda^{n,\alpha}(t,\theta^*)-\lambda^{\infty,\alpha}(t,\theta^*)\big|dt
\nn\\&&
\qquad+{\blue \sum_{j=0}^1}
\int_{T_0}^{T_1} \big|{\blue\partial_\theta^j}\lambda^{\infty,\alpha}(t,\theta)-{\blue\partial_\theta^j}\lambda^{\infty,\alpha}(t,\theta^*)\big|dt\bigg\}
\nn\\&&
\eea
{\blue Under $[A1]_1$, $[A2]_1$ and $[A_3]_1$, 
Lemma \ref{260808-3} gives 
\beas 
\int_{T_0}^{T_1}\big|
\partial_\theta^j\lambda^{n,\alpha}(t,\theta)-\partial_\theta^j\lambda^{\infty,\alpha}(t,\theta)
\big|dt
&\to^p& 0
\eeas
as $n\to\infty$ for each $\theta\in\Theta$ for $j=0,1$. 
Then with the the tightness of 
$\{\sup_{t,\theta}|\partial_\theta^j\lambda^n(t,\theta)|\}_{n\in\bbN}$ for $j=1,2$, 
deduced from $[A1]_2$ and $[A2]_2$,}
we obtain the uniform convergence 
\bea\label{260810-10}
\sup_{\theta\in {\colorc V_n}}
\sum_{j=0}^1\int_{T_0}^{T_1} \big|\partial_\theta^j\lambda^{n,\alpha}(t,\theta)
-\partial_\theta^j\lambda^{\infty,\alpha}(t,\theta)\big|dt
&\to^p& 0
\eea
Moreover, 
\bea\label{260810-9}
\lim_{\theta\to\theta^*}
\int_{T_0}^{T_1}\big|{\blue \partial_\theta^j}
\lambda^\infty(t,\theta)-{\blue \partial_\theta^j}\lambda^\infty(t,\theta^*)\big|dt
&=&0\quad a.s.\quad{\blue (j=0,1)}
\eea
by the regularity of $\lambda^\infty$. 
{\blue Indeed, this property follows from (\ref{260810-10}) and 
tightness of 
$\{\sup_{t,\theta}{\colorc|}\partial_\theta^j\lambda^n(t,\theta){\colorc|}\}_{n\in\bbN}$ for 
$j=1,2$. }
Then {\colorb Lemma \ref{260617-1},} (\ref{260716-1}), (\ref{260810-10}) and (\ref{260810-9}) yield the lemma 
with an argument with localization. 
\qed\halflineskip

{\blue 
We shall recall a mixed normal limit theorem. 
{\colorc Given} 
a stochastic basis $(\Omega,\calf,{\bf F},P)$ with ${\bf F}=(\calf_t)_{t\in[T_0,T_1]}$,
we suppose that $\mu^{n,\alpha}$ ($\alpha=1,...,{\sf d})$ are integer-valued 
random measures on ${\sf E}_\alpha=\bbR^{{\sf d}_\alpha}\setminus\{0\}$ with 
compensators $\nu^{n,\alpha}$ respectively. 
Let $c^n_\alpha:\Omega\times \bbR^{{\sf d}_\alpha}\to\bbR^{\sf d}$ be predictable 
processes. 
Let 
\bea\label{271106-1} 
L^n_t &=& \sum_\alpha\int_{T_0}^t\int_{{\sf E}_\alpha} 
c^n_\alpha(t,x)\tilde{\mu}^{n,\alpha}(dt,dx),
\eea
where 
$\tilde{\mu}^{n,\alpha}=\mu^{n,\alpha}-\nu^{n,\alpha}$. 
%
\begin{lemma*}\label{270730-1}
Suppose that the following conditions are fulfilled. 
\bi 
\im[{\rm (i)}] For each $n\in\bbN$, any two of $\mu^{n,\alpha}$ do not have 
common jump times and $\nu^{n,\alpha}(\{s\},{\sf E}_\alpha)=0$ for all $s\in[T_0,T_1]$. 
\im[{\rm (ii)}] There exists an $\bbR^{\sfd}\otimes\bbR^{\sf r}$-valued ${\bf F}$-predictable process $g=(g_t)$ 
such that $\int_{T_0}^{T_1}|g_s|^2ds<\infty$ a.s. and that 
\beas 
\sum_\alpha\int_{T_0}^t\int_{{\sf E}_\alpha}  c^n_\alpha(s,x)^{\otimes2}\nu^{n,\alpha}(ds,dx)
&\to^p& 
\int_{T_0}^t g_s^{\otimes2} ds
\eeas
as $n\to\infty$ for every $t\in[T_0,T_1]$. 

\im[{\rm (iii)}] For every $\ep>0$,  
\beas 
\int_{T_0}^{T_1}\int_{\{|c^n_\alpha(s,x)|>\ep\}}
 |c^n_\alpha(s,x)|^2\nu^{n,\alpha}(ds,dx) &\to^p& 0
\eeas
as $n\to\infty$
\ei
Then $L^n\to^{d_s}\int_{T_0}^\cdot g_s dW_s$ in $\bbD([T_0,T_1];\bbR^\sfd)$, 
where $W$ is an ${\sf r}$-dimensional standard Wiener process 
$($defined on an extension of $(\Omega,\calf,P)$$)$ 
independent of $g$. 
\end{lemma*}
We leave a sketch of proof for reader's convenience in Appendix. 
}

%
\begin{lemma*}\label{260716-2} 
$n^{-1/2}\partial_\theta\ell_n(\theta^*)\to^{d_s} \Gamma^{1/2}\zeta$ as $n\to\infty$, 
where $\zeta$ is a $\sfp$-dimensional standard Gaussian random vector defined on an extended 
probability space of $(\Omega,\calf,P)$ and independent of $\calf$, and 
$d_s$ denotes the $\calf$-stable convergence. 
\end{lemma*}
\proof 
{\colorc
Let 
$\mu^{n,\alpha}(dt,dx) = N^{n,\alpha}(dt)\times\delta_1(dx)$, 
$\nu^{n,\alpha}(dt,dx) =n \lambda^{n,\alpha}(t,\theta^*)dt\>\delta_1(dx)$ and 
\beas 
c^n_\alpha(t,x) &=& 
n^{-1/2} \lambda^{n,\alpha}(t,\theta^*)^{-1}
\partial_\theta\lambda^{n,\alpha}(t,\theta^*).
\eeas
Then $L^n_t$ in (\ref{271106-1}) has an expression 
\beas 
L^n_t &=& 
n^{-1/2}\sum_\alpha \int_{t\wedge T_0}^t \lambda^{n,\alpha}(t,\theta^*)^{-1}
\partial_\theta\lambda^{n,\alpha}(t,\theta^*)
d\tilde{N}^{n,\alpha}_t
\eeas
for $t\in \hat{I}$. 
Following the argument in the proof of Lemma \ref{260618-1}, we see 
the convergence in (ii) of Lemma \ref{270730-1} holds for 
\beas 
\int_{T_0}^tg_s^{\otimes2}ds
&=&
\int_{T_0}^t
(\partial_\theta\lambda^{\infty,\alpha})^{\otimes2}(s,\theta^*)\big(\lambda^{\infty,\alpha}(s,\theta^*)\big)^{-1}
ds.
\eeas
Tightness of the family 
\beas
\bigg\{
\sum_{j=0}^1\sup_{t\in[T_0,T_1]}|\partial_\theta^j\lambda^{n,\alpha}(t,\theta^*)|,\>
\sup_{t\in[T_0,T_1]}\lambda(t,\theta^*)^{-1}1_{\{\lambda^{n,\alpha}(t,\theta^*)\not=0\}};\>
\alpha\in\cali,\>n\in\bbN\bigg\}
\eeas
verifies (iii) of Lemma \ref{270730-1} via the Lyapunov condition.
}
\begin{en-text}
Let ${\sf M}$ be an $\F$-locally square-integrable (purely discontinous) martingale with 
the canonical representation 
\beas 
{\sf M}_t &=& B_t+\int_{\hat{T}_0}^t \int z1_{\{|z|\leq1\}}\>(\mu-\nu)(ds,dz),
\eeas
where $B_t$ is a bounded variational process and $\mu$ is an integer-valued random measure 
compensated by $\nu$. 
Extend $\tilde{N}^{n,\alpha}$ to $[\hat{T}_0,T_1]$ by $\tilde{N}^{n,\alpha}_{t\wedge T_0}$. 
For $\ep>0$, let 
\beas 
{\sf M}^{\underline{\ep}} &=& B_t + \int_{\hat{T}_0}^t \int z1_{\{\ep<|z|\leq1\}}\>(\mu-\nu)(ds,dz)
\eeas
and let 
\beas 
{\sf M}^{\overline{\ep}} &=& \int_{\hat{T}_0}^t \int z1_{\{|z|\leq\ep\}}\>(\mu-\nu)(ds,dz).
\eeas
Then
\beas 
{\rm Var}\l[n^{-1/2}\tilde{N}^{n,\alpha},{\sf M}\r]
&=&
{\rm Var}\l[n^{-1/2}\tilde{N}^{n,\alpha},{\sf M}^{\underline{\ep}}\r]
+{\rm Var}\l[n^{-1/2}\tilde{N}^{n,\alpha},{\sf M}^{\overline{\ep}}\r]
\\&\leq&
n^{-1/2}\sum_{s\leq\cdot}\Delta N^{n,\alpha}_s|\Delta{\sf M}^{\underline{\ep}}_s|
\>1_{\{|\Delta{\sf M}^{\underline{\ep}}_s|>\ep\}}
+\l(n^{-1}N^{n,\alpha}\r)^{1/2}\l[{\sf M}^{\overline{\ep}},{\sf M}^{\overline{\ep}}\r]^{1/2}
\\&\leq&
n^{-1/2}\l(2\sup_{s\leq T_1} |{\sf M}_s|\r)\>\#\{s\leq T_1;\>|\Delta{\sf M}^{\underline{\ep}}_s|>\ep\}
\\&&
+\l(n^{-1}N^{n,\alpha}_{T_1}\r)^{1/2}\l[{\sf M}^{\overline{\ep}},{\sf M}^{\overline{\ep}}\r]_{T_1}^{1/2}.
\eeas
By definition, $\lim_{\ep\down0}\l[{\sf M}^{\overline{\ep}},{\sf M}^{\overline{\ep}}\r]_{T_1}=0$ a.s. 
Therefore 
${\rm Var}[n^{-1/2}\tilde{N}^{n,\alpha},{\sf M}]_{T_1} \to^p0$, and hence 
${\rm Var}[L^n,{\sf M}]_{T_1} \to^p0$, consequently 
$\langle L^n,{\sf M}\rangle_t \to^p0$
 as $n\to\infty$ for every $t$. 
 Obviously,  
$\langle L^n,{\sf M}\rangle=0$ for $\F$-continuous martingales ${\sf M}$. 
\end{en-text}
\begin{en-text}
We will applying Lemma \ref{270730-1} to 
\beas 
c^n_\alpha(t,x) 
&=&
\eeas

By Lemma \ref{251123-1} together with the orthogonality between $\tilde{N}^{n,\alpha}$ $(\alpha\in\cali)$, 
and by Lemma {\colorb \ref{260808-3}}, 
we obtain  
{\colorb 
\beas 
\langle n^{-1/2}\tilde{N}^{n,\alpha},n^{-1/2}\tilde{N}^{n,\alpha'}\rangle_t 
&=&
\delta_{\alpha,\alpha'}
 \int_{T_0}^t \lambda^{n,\alpha}(s,\theta^*)^2ds
\>\to^p\>
\delta_{\alpha,\alpha'}
 \int_{T_0}^t \lambda^{\infty,\alpha}(s,\theta^*)^2ds
\eeas
uniformly in $t$ as $n\to\infty$. 
}
As a result, 
\beas 
\langle L^n,L^n\rangle_t 
&\to^p& 
\sum_\alpha \int_{t\wedge T_0}^t \lambda^{\infty,\alpha}(t,\theta^*)^{-1}
\l(\partial_\theta\lambda^{\infty,\alpha}(t,\theta^*)\r)^{\otimes2}
dt
\eeas
as $n\to\infty$ for every $t$. 
Obviously, $\sum_{s:s\leq T_1}1_{\{|\Delta L^n_s|>\ep\}}{\colorb \to^p}0$ for large $n\to\infty$, for every $\ep>0$. 
\end{en-text}
Thus, {\colorb by Lemma \ref{270730-1}, }
we obtained the stable convergence of $n^{-1/2}\partial_\theta\ell_n(\theta^*)$. 
\qed \halflineskip

\begin{theorem*}\label{260716-4}
Suppose that $[A1]_3$, $[A2]_3$, $[A3]_{\colorrz 1}$ and ${\colorr [A4]}$ are fulfilled. 
Suppose that the estimator $\hat{\theta}_n$ of $\theta$ satisfies 
$\hat{\theta}_n\to^p\theta^*$ and $\partial_\theta\ell_n(\hat{\theta}_n)=o_p(n^{1/2})$ 
as $n\to\infty$. 
Then 
\beas 
\sqrt{n}\big(\hat{\theta}_n-\theta^*\big) 
&\to^{d_s}& 
\Gamma^{-\half}\zeta
\eeas
as $n\to\infty$, 
where $\zeta$ is a $\sfp$-dimensional standard Gaussian random vector given in Lemma \ref{260716-2}. 
\end{theorem*}
\proof It is easy to obtain the result from Lemmas \ref{260618-1} and \ref{260716-2}. 
{\colorb 
Indeed, there is a sequence $\{V_n\}_{n\in\bbN}$ of open balls centered at $\theta^*$ 
such that the diameter of $V_n$ tends to $0$ as $n\to\infty$ and $P[\hat{\theta}_n\in V_n]\to1$. 
On the event $\{\hat{\theta}_n\in V_n\}$, one has 
\beas 
n^{-1/2}\partial_\theta\ell_n(\hat{\theta}_n)[u]-n^{-1/2}\partial_\theta\ell_n(\theta^*)[u]
&=&
n^{-1}\int_0^1 \partial_\theta^2\ell_n(\theta_n(s))ds\big[n^{1/2}(\hat{\theta}_n-\theta^*),u\big]
\eeas
for $u\in\bbR^{\sf p}$, where $\theta_n(s)=\theta^*+s(\hat{\theta}_n-\theta^*)$. 
Then the result is easily obtained. 
}
\qed\halflineskip

\section{The quasi likelihood analysis: QMLE and QBE}\label{qla-pp}

In Theorem \ref{260716-4}, we obtained a limit theorem for the quasi maximum likelihood estimator $\hat{\theta}_n$. 
It is the first step of analysis of the estimator, however, more precise estimates 
for the tail of the distribution of the estimator will be indispensable to develop basic theory of statistical inference 
such as asymptotic decision theory, prediction, higher-order efficiency, information criteria, etc. 
This section presents the so-called quasi likelihood analysis, that gives certain tail probability estimates 
of the quasi {\colorb maximum} likelihood estimator (QMLE) and the quasi Bayesian estimator (QBE). 

\subsection{Polynomial type large deviation inequality for the quasi likelihood random field}
We shall work with the statistical random field
\beas 
\bbH_n(\theta) &=& \ell_n(\theta)
\eeas
on $\Theta$ and apply the frame of the quasi likelihood analysis in \cite{Yoshida2011}. 
The random fields $\bbZ_n$ is defined on $\bbU_n=\{u\in\bbR^{\sfp};\>\theta_u\in\Theta\}$, 
$\theta_u=\theta^*+n^{-1/2}u$, by 
\beas 
\bbZ_n(u)
&=&
\exp\big(
\bbH_n(\theta_u)-\bbH_n(\theta^*)
\big)
\\&=&
\exp\bigg(
\sum_{\alpha=1}^\sfd
\int_{T_0}^{T_1} \log\frac{\lambda^{n,\alpha}(t,\theta_u)}
{\lambda^{n,\alpha}(t,\theta^*)}
\>dN^{n,\alpha}_t
\\&&
-\sum_{\alpha=1}^\sfd\int_{T_0}^{T_1}{\colorg n}
\big[\lambda^{n,\alpha}(t,\theta_u)-\lambda^{n,\alpha}(t,\theta^*)\big]dt
\bigg).
\eeas

%

%
%
Under necessary regularity conditions specified later, 
we define a random vector $\Delta_n$ by 
\beas 
\Delta_n[u] 
&=&
\frac{1}{\sqrt{n}}
\sum_{\alpha=1}^\sfd \int_{T_0}^{T_1}
\lambda^{n,\alpha}(t,\theta^*)^{-1}
\partial_\theta{\lambda}^{n,\alpha}(t,\theta^*)[u]
d\tilde{N}^{n,\alpha}_t
\eeas
and a random matrix $\Gamma_n(\theta)$ by 
\beas 
\Gamma_n(\theta) [u^{\otimes2}]
&=&
-\partial_\theta^2\bbH_n(\theta)[(n^{-1/2}u)^{\otimes2}]
\eeas
for $u\in\bbR^\sfp$.

Let 
\beas 
r^{(1)}_n(u) 
&=&
\sum_{\alpha=1}^\sfd \int_{T_0}^{T_1}
\bigg\{\log\frac{\lambda^{n,\alpha}(t,\theta_u)}{\lambda^{n,\alpha}(t,\theta^*)}
-{\lambda}^{n,\alpha}(t,\theta^*)^{-1}
\partial_\theta{\lambda}^{n,\alpha}(t,\theta^*)[n^{-1/2}u]
\bigg\}d\tilde{N}^{n,\alpha}_t,
\eeas
and 
\beas 
r^{(2)}_n(u) &=& {\colorg -}
\sum_{i=1}^3r^{(2i)}_n(u),
\eeas
where 
\beas 
r^{(21)}_n(u)
&=&
\sum_{\alpha=1}^\sfd
\int_{T_0}^{T_1}\bigg\{
\frac{\lambda^{n,\alpha}(t,\theta_u)}{\lambda^{n,\alpha}(t,\theta^*)}-1
-\log\frac{\lambda^{n,\alpha}(t,\theta_u)}{\lambda^{n,\alpha}(t,\theta^*)}
\\&&\qquad
-\half\bigg[\frac{\lambda^{n,\alpha}(t,\theta_u)}{\lambda^{n,\alpha}(t,\theta^*)}-1
\bigg]^2
\bigg\}n\lambda^{n,\alpha}(t,\theta^*)dt
\eeas
\beas 
r^{(22)}_n(u)
&=&
\sum_{\alpha=1}^\sfd
\half\int_{T_0}^{T_1}\bigg\{
\bigg[\lambda^{n,\alpha}(t,\theta_u)-\lambda^{n,\alpha}(t,\theta^*)
\bigg]^2 
\\&&\qquad
-
\big(\partial_\theta{\lambda}^{n,\alpha}(t,\theta^*)[n^{-1/2}u]
\big)^2 \bigg\}
n{\lambda}^{n,\alpha}(t,\theta^*)^{-1}dt,
\eeas
\beas 
r^{(23)}_n(u)
&=&
\sum_{\alpha=1}^\sfd
\half\int_{T_0}^{T_1}
\bigg\{
{\lambda}^{n,\alpha}(t,\theta^*)^{-1}\big(\partial_\theta{\lambda}^{n,\alpha}(t,\theta^*)[u]\big)^2
\\&&
-
{\lambda}^{\infty,\alpha}_0(t)^{-1}\big({\lambda}_1^{\infty,\alpha}(t)[u]\big)^2\bigg\}dt.
\eeas
\begin{en-text}
with 
\beas 
\lambda^{\infty}_0(t) 
&=& 
g_0^\infty(t)+\int_{\hat{T}_0}^t K_0^\infty(t,s)dX^\infty_s
\eeas
and 
\beas 
\lambda^{\infty}_1(t) 
&=& 
g_1^\infty(t)+\int_{\hat{T}_0}^t K_1^\infty(t,s)dX^\infty_s. 
\eeas
\end{en-text}
Let $r_n(u)=r^{(1)}_n(u)+r^{(2)}_n(u)$. 
\begin{en-text}
Moreover, let 
\beas 
\Gamma
&=&
\sum_\alpha
\int_{T_0}^{T_1}
\l[
{\lambda}_1^{\infty,\alpha}(t)
\r]^{\otimes2}
{\lambda}_0^{\infty,\alpha}(t)^{-1}dt.
\eeas
\end{en-text}
Then we have an expression of $\bbZ_n$ as 
\bea\label{lamn} 
\bbZ_n(u)
&=&
\exp\bigg(
\Delta_n[u]-\half\Gamma[(u)^{\otimes2}]+r_n(u)\bigg)
\eea
for $u\in\bbU_n$, which suggests the LAMN property of the random field $\bbH_n$.

Assume the condition (\ref{260810-1}) for $\Theta$. 
Moreover we suppose that 
the function $\Theta\ni\theta\mapsto\lambda^n(t,\theta)$ has continuous extension to $\bar{\Theta}$ 
when the QMLE is discusses.

Let $\varepsilon$ be a positive number less than $1/2$. 
Let $\bar{\bbN}=\bbN\cup\{\infty\}$. 
For mathematical validation of the asymptotic properties deduced below, we shall assume the following conditions. 

%
\bd
\item[[B1\!\!]]$_{\bar{j}}$
For each $n\in\bar{\bbN}$, 
$K^n(t,s,\theta)$ is an $\bbR_+^\sfd\otimes\bbR_+^{\sfd_0}$-valued  
$\calf\times\bbB(J)\times\bbB(\Theta)$-measurable 
function satisfying the following conditions. 
\bd
\im[(i)] 
For each $(n,t,\theta)\in \bbN\times I\times\Theta$, the process 
$[\hat{T}_0,t)\ni s\mapsto K^n(t,s,\theta)$ is 
$(\calf_s)_{s\in[\hat{T}_0,t)}$-optional. 

\im[(ii)] 
For each $(n,t,s)\in\bar{\bbN}\times J$, the mapping 
$\Theta\ni\theta\mapsto K^{n}(t,s,\theta)$ is $\bar{j}$ times differentialble a.s.,
$\sup_{(s,\theta)\in[\hat{T}_0,t)\times\Theta}|\partial_\theta^j K^n(t,s,\theta)|<\infty$ a.s. 
for $t\in I$, 
and 
\beas 
\sum_{j=0}^{\bar{j}}
\sup_{(n,s,t)\in \bar{\bbN}\times J}\sup_{\theta\in\Theta} \|\partial_\theta^iK^n(t,s,\theta)\|_p 
&<&
\infty
\eeas
for every $p>1$. 

\im[(iii)] 
For each {\colorb $(n,t)\in\bbN\times I$}, the mappings 
$[\hat{T}_0,t)\ni s\mapsto \partial_\theta^i K^n(t,s,\theta^{\colorb *})$ ($i=0,1$)
are differentialble a.s., 
$\sup_{(s,\theta)\in[\hat{T}_0,t)\times\Theta}| \partial_s\partial_\theta^i K^n(t,s,\theta^{\colorb *})|<\infty$ a.s. 
for $t\in I$, 
and 
\beas 
\sup_{{\colorb (n,t)\in \bbN\times I}} 
\sum_{i=0}^1
\int_{\hat{T}_0}^t\|  \partial_s\partial_\theta^iK^n(t,s,\theta^{\colorb *})\|_pds
&<& \infty
\eeas
for every $p>1$.

\im[(iv)] 
For every $p>1$, 
\beas 
n^\varepsilon
\sum_{j=0}^1
\sup_{{\colorb (t,\theta)\in I\times\Theta}}
\l\|{\colorb \sup_{s\in[\hat{T}_0,t)}\big|}
\partial_{\theta}^jK^{n}(t,s,\theta)
-\partial_{\theta}^jK^{\infty}(t,s,\theta)
{\colorb \big|}
\r\|_p
\to0
\eeas
as $n\to\infty$.
 
\ed
\ed
\begin{description}
\item[[B2\!\!]]$_{\bar{j}}$ 
For each $(\alpha,n)\in\cali\times\bar{\bbN}$, 
$g^{n,\alpha}(t,\theta)$ is an 
nonnegative 
$\calf\times\bbB(I)\times\bbB(\Theta)$-measurable function 
for which the following conditions are fulfilled. 
\bd
\im[(i)] For each $(n,\alpha,\theta)\in\bbN\times\cali\times\Theta$, 
the process $(g^{n,\alpha}(t,\theta))_{t\in I}$ is predictable.  

\im[(ii)] For each $(n,t)\in\bar{\bbN}\times I$, the mapping $\Theta\ni\theta\mapsto g^{n}(t,\theta)$ is 
$\bar{j}$ times differentiable a.s. and 
\beas 
\sum_{j=0}^{\bar{j}} \sup_{(n,t)\in\bar{\bbN}\times I}\sup_{\theta\in\Theta}
\big\|(\partial_\theta)^j g^n(t,\theta)\big\|_p
&<&\infty
\eeas
for every $p>1$. 

\im[(iii)] 
For every $p>1$, 
\beas 
n^\varepsilon
\sum_{j=0}^1\sup_{t\in I}\sup_{\theta\in\Theta}
\l\| \partial_\theta^jg^{n}(t,\theta)
-\partial_\theta^jg^{\infty}(t,\theta)\r\|_p
&\to&0
\eeas
as $n\to\infty$. 
\ed
\ed

\bd
\item[[B3\!\!]] 
For each {\colorb $(\beta,n)\in\cali_0\times{\colorc\bar{\bbN}}$}, $(X^{n,\beta}_t)_{t\in\hat{I}}$ is a non-decreasing 
$(\calf_t)_{t\in\hat{I}}$-adapted process
such that 
\beas 
\sup_{(n,t)\in \bbN\times\hat{I}}\big\|X^{n,\beta}_t\big\|_p
\><\>\infty\quad\mbox{and}\quad
n^\varepsilon
\sup_{t\in \hat{I}}
\big\|X^{n,\beta}_t
-X^{\infty,\beta}_t
\big\|_p
\>\to\>0
\eeas
as $n\to\infty$,  
for every $p>1$. 
($\hat{I}=[\hat{T}_0,T_1]$.)
\ed

{\colorr 
\bd
\item[[B4\!\!]] 
For each $(\omega,n,\alpha,t,\theta)\in\Omega\times\bbN\times\cali\times I\times\Theta$, 
$\lambda^{n,\alpha}(t,\theta)=0$ if and only if $\lambda^{n,\alpha}(t,\theta^{\colorb *})=0$, 
and 
\beas 
\sup_{(n,t,\theta)\in I\times\Theta}\big\|\lambda^{n,\alpha}(t,\theta)^{-1}
{\colorr 1_{\{\lambda^{n,\alpha}(t,\theta)\not=0\}}}
\big\|_p 
&<& \infty
\eeas
for every $p>1$ and $\alpha\in\cali$. 
\ed
\halflineskip
}

Define the index $\chi_0$ by 
\beas 
\chi_0 &=& 
\inf_{\theta\in\Theta\setminus\{\theta^*\}}\frac{-\bbY(\theta)}{|\theta-\theta^*|^2},
\eeas
where $\bbY$ is given in (\ref{260810-11}). 
\halflineskip
The nondegeneracy of the key index $\chi_0$ will play an essential role in our argument. 
\bd
\item[[{\colorr B5}\!\!]]
For every $L>0$, there exists a constant $C_L$ such that 
\beas 
P[\chi_0<r^{-1}] &\leq& \frac{C_L}{r^L}
\eeas
for all $r>0$. 
\ed

\begin{lemma*}\label{260730-3}
{\colorb Suppose that Conditions $[B1]_4$, $[B2]_4$, $[B3]$ and $[B4]$ are satisfied. Then}
\beas 
\sup_{n\in\bbN} \bigg\|n^{-1}\sup_{\theta\in\Theta}\big|\partial_\theta^3\bbH_n(\theta)\big|\bigg\|_p
&<& \infty 
\eeas
for every $p>1$. 
\end{lemma*}
\proof
{\colorb
We have a representation 
\beas 
\partial_\theta^j\lambda^n(t,\theta) 
&=& 
\partial_\theta^jg^n(t,\theta)
+ \int_{\hat{T}_0}^{t-} \partial_\theta^jK^n(t,s,\theta)dX^n_s.
\eeas
Like (\ref{260809-3}), 
use Sovolev's inequality and the Burkholder-Davis-Gundy inequality to obtain the desired estimate. 
}
\begin{en-text}
With the representation 
\beas 
\partial_\theta^j\lambda^n(t,\theta) -\partial_\theta^j\lambda^\infty(t,\theta)
&=& 
\partial_\theta^jg^n(t,\theta)-\partial_\theta^jg^\infty(t,\theta)
\\&&
+ \int_{\hat{T}_0}^{t-} \partial_\theta^jK^n(t,s,\theta)dX^n_s 
- \int_{\hat{T}_0}^{t-} \partial_\theta^jK^\infty(t,s,\theta)dX^\infty_s 
\eeas
for $j=0,...,4$, 
we have 
\beas 
\sup_{(n,\theta)\in\bbN\times\Theta}\sum_{j=0}^4
\big\|\partial_\theta^j\lambda^n(t,\theta) -\partial_\theta^j\lambda^\infty(t,\theta)\big\|_p
&<&
\infty
\eeas
for every $p>1$. 
Use Sovolev's inequality and the Burkholder-Davis-Gundy inequality to obtain the desired estimate. 
\end{en-text}
\qed

\begin{lemma*}\label{260730-4}
For every $p>1$, 
\beas 
\sup_{n\in\bbN} \big\|n^\varepsilon|\Gamma_n(\theta^*)-\Gamma|\big\|_p
&<& \infty. 
\eeas
\end{lemma*}
\proof 
We have the representation
\bea\label{260731-1}
\partial_\theta^j\lambda^n(t,\theta) -\partial_\theta^j\lambda^\infty(t,\theta)
\nn&=& 
\partial_\theta^jg^n(t,\theta)-\partial_\theta^jg^\infty(t,\theta)
\nn\\&&
+
\int_{\hat{T}_0}^{t-} (\partial_\theta^jK^n(t,s,\theta)-\partial_\theta^jK^\infty(t,s,\theta))dX^\infty_s 
\nn\\&&
+\partial_\theta^jK^n(t,t-,\theta)\big(X^n_{t-}-X^\infty_{t-}\big)
-\partial_\theta^jK^n(t,\hat{T}_0,\theta)\big(X^n_{\hat{T}_0}-X^\infty_{\hat{T}_0}\big)
\nn\\&&
-\int_{\hat{T}_0}^{t-} (\partial_s\partial_\theta^jK^n)(t,s,\theta^*)\big(X^n_s- X^\infty_s\big)ds
\eea
for $j=0,1$ and $\theta\in\Theta$. 
{\colorb Then it is possible to obtain the desired estimate 
by using (\ref{260716-1}), 
(\ref{260731-1}) evaluated at $\theta=\theta^*$ and an estimate similar to (\ref{260809-2}). 
}
\qed\halflineskip

{\colorb Obviously we have }
\begin{lemma*}\label{260730-5}
For every $p>1$, 
$
\sup_{n\in\bbN} \big\|{\colorb \Delta_n}\big\|_p
< \infty. 
$
\end{lemma*}

\begin{lemma*}\label{260730-6}
For every $p>1$, 
\beas 
\sup_{n\in\bbN} \bigg\|n^\varepsilon\sup_{\theta\in\Theta}\big|\bbY_n(\theta)-\bbY(\theta)\big|\bigg\|_p
&<& \infty. 
\eeas
\end{lemma*}
\proof 
Use (\ref{260731-1}) and Sobolev's inequality, as well as the Burkholder-Davis-Gundy inequality 
for uniform estimate of the martingale part. 
\qed\halflineskip

\begin{proposition*}\label{260730-2}
{\rm (Polynomial type large deviation inequality)} 
Suppose that Conditions $[B1]_4$, $[B2]_4$, $[B3]$, $[B4]$ {\colorr and $[B5]$} are fulfilled. 
Then, for every $L>0$, there exists a constant $C_L$ such that 
\beas 
P\bigg[\sup_{u\in\bbV_n(r)}\bbZ_n(u)\geq e^{-r}\bigg] 
&\leq& 
\frac{C_L}{r^L}
\eeas
for all $r>0$ and all $n\in\bbN$, where 
$\bbV_n(r)=\{u\in\bbU_n;\>|u|\geq r\}$. 
\end{proposition*}
\proof 
We will follow the procedure in \cite{Yoshida2011}. 
The parameters there, here in quotes, will be set as follows: 
let $\mbox{``}\rho\mbox{''} = 2$, 
$\mbox{``}\beta_1\mbox{''}=\varepsilon$ and 
$\mbox{``}\beta_2\mbox{''}=\half-\varepsilon$, next choose  
$\mbox{``}\rho_2\mbox{''}\in(0,2\varepsilon)$, 
take $\mbox{``}\alpha\mbox{''}\in(0,\rho_2/2)$, 
finally take $\mbox{``}\rho_1\mbox{''}\in(0,\min\{1,\alpha/(1-\alpha),2\varepsilon/(1-\alpha)\}$. 
Then Condition [A4$'$] of \cite{Yoshida2011} is satisfied. 
\begin{en-text}
\beas 
\mbox{``}\rho\mbox{''} = 2,\quad
\mbox{``}\beta_1\mbox{''}=\frac{1}{4},\quad
\mbox{``}\beta_2\mbox{''}=0,\quad
\mbox{``}\rho_2\mbox{''}=\frac{3}{4},\quad
\mbox{``}\alpha\mbox{''}=\frac{1}{4},\quad
\mbox{``}\beta\mbox{''}=\frac{1}{3},\quad
\mbox{``}\rho_1\mbox{''}=\frac{1}{7}.
\eeas
\end{en-text}
Lemmas \ref{260730-3} and \ref{260730-4} ensure [A1$''$] of \cite{Yoshida2011}. 
Condition {\colorr [B5]} implies Conditions [A2] and [A5] of \cite{Yoshida2011}. 
Condition [A3] of \cite{Yoshida2011} obviously holds. 
Condition [A6] of \cite{Yoshida2011} is checked by Lemmas \ref{260730-5} and \ref{260730-6}. 
Apply Theorem 2 of \cite{Yoshida2011} to obtain the result. 
\qed\halflineskip

\subsection{Limit distribution and  moment convergence of QMLE and QBE}

A quasi maximum likelihood estimator is any estimator that satisfies 
\beas 
\bbH_n(\hat{\theta}_n) &=& \max_{\theta\in\bar{\Theta}}\bbH_n(\theta). 
\eeas
The quasi Bayesian estimator is defined by 
\beas 
\tilde{\theta}_n &=& 
\bigg[\int_{\Theta}\exp(\bbH_n(\theta))\>\varpi(\theta)d\theta\bigg]^{-1}
\int_{\Theta}\theta\exp(\bbH_n(\theta))\>\varpi(\theta)d\theta
\eeas
for a prior density $\varpi$ on $\Theta$. We will assume that 
$\varphi$ is continuous and 
$0<\inf_{\theta\in\Theta}\varpi(\theta)\leq\sup_{\theta\in\Theta}\varpi(\theta)<\infty$. 

Let 
\beas 
\bbZ(u) &=& \exp\bigg(\Delta[u]-\half\Gamma\>[u^{\otimes2}]\bigg)
\eeas
for $u\in\bbR^\sfp$, 
where $\Delta=\Gamma^{1/2}\>\zeta$. 

Denote by $C_\up(\bbR^\sfp)$ the set of continuous functions $f:\bbR^\sfp\to\bbR$ 
{\colorb of} at most polynomial growth. 
\begin{theorem*}\label{260731-2} 
Suppose that Conditions $[B1]_4$, $[B2]_4$, $[B3]$, $[B4]$ {\colorr and $[B5]$} are fulfilled. 
Then 
\bi
\im[$(a)$] 
$\displaystyle\sqrt{n}(\hat{\theta}_n-\theta^*)\to^{d_s}\Gamma^{-1/2}\zeta$ as $n\to\infty$. 
\im[$(b)$] 
$\displaystyle 
E\big[f\big(\sqrt{n}(\hat{\theta}_n-\theta^*)\big)\big]\to 
\bbE\big[f\big(\Gamma^{-1/2}\zeta\big)\big]$ as $n\to\infty$ 
for all $f\in C_\up(R^\sfp)$. 
\ei
\end{theorem*}
\halflineskip
\begin{theorem*}\label{260731-3} 
Suppose that Conditions $[B1]_4$, $[B2]_4$, $[B3]$, $[B4]$ {\colorr and $[B5]$} are fulfilled. 
Then 
\bi
\im[$(a)$] 
$\displaystyle\sqrt{n}(\tilde{\theta}_n-\theta^*)\to^{d_s}\Gamma^{-1/2}\zeta$ as $n\to\infty$. 
\im[$(b)$] 
$\displaystyle 
E\big[f\big(\sqrt{n}(\tilde{\theta}_n-\theta^*)\big)\big]\to 
\bbE\big[f\big(\Gamma^{-1/2}\zeta\big)\big]$ as $n\to\infty$ 
for all $f\in C_\up(R^\sfp)$. 
\ei
\end{theorem*}
\halflineskip

\noindent 
{\it Proof of Theorems \ref{260731-2} and \ref{260731-3}.} 
Lemma \ref{260716-2} implies the finite-dimensional stable convergence of 
$\bbZ_n\to^{d_{f,s}}\bbZ$ as $n\to\infty$ [the proof is still valid under the present assumptions]. 
Since 
\beas 
\sum_{j=0}^3\sup_{(n,t)\in\bbN\times I}\|\partial_\theta^j\lambda^n(t,\theta)\|_p 
&<&\infty
\eeas
for any $p>1$, 
we obtain the tightness of the residual random fields $\{r_n|_K\}_{n\in\bbN}$ 
restricted to $K$, and hence that of 
$\{\bbZ_n|_K\}_{n\in\bbN}$ in $C(K)$ for every compact set $K$ in $\bbR^\sfp$. 
Consequently, $\bbZ_n|_K\to^{d_s}\bbZ|_K$ in $C(K)$ for compacts $K$. 
Now 
Theorem 4 of \cite{Yoshida2011} gives the proporties of the QMLE. 
Moreover, by Lemma 2 of \cite{Yoshida2011}, 
\beas 
\sup_{n\in\bbN}E\bigg[\bigg(\int_{u:|u|\leq\delta}\bbZ_n(u)\bigg)^{-1}\bigg] &<& \infty.
\eeas
Then Theorem 8 of \cite{Yoshida2011} provides the asymptotic properties of the QBE. 
\qed\halflineskip

\begin{en-text}
\subsection{Localization (especially for QBE)}\koko
By localization, we may assume 
For each $(n,t,s)\in\bbN\times J$, the mapping 
$\Theta\ni\theta\mapsto K^{n,\alpha}_\beta(t,s,\theta)$ is three times differentialble a.s. and 
\beas 
\sum_{j=0}^3
\sup_{n\in\bbN}
\sup_{(t,s)\in J}\sup_{\theta\in\Theta} \|(\partial_\theta)^j K^n(t,s,\theta)\|_q<\infty
\quad 
\eeas
for every $q>1$. 
\end{en-text}

\section{Hawkes type processes}

In this section, the point process regression model will be applied to a feedback system of point processes. 
We will consider the explanatory variables $X^n=n^{-1}N^n$, that is, 
the the process $\lambda^n(t,\theta)$ will be 
\bea\label{hawkes.type}
\lambda^n(t,\theta) &=& g^n(t,\theta)+\int_{\hat{T}_0}^{t-}K^n(t,s,\theta)n^{-1}dN^n_s
\eea
for $t\in I$. 
The Hawkes process is a special case of this model. 
It should be remarked that 
Hawkes processes are often used to describe ergodic systems in long-run, whereas 
we will work with non-ergodic processes with finite time horizon and the intensities diverge.  

Here a QLA will be formulated according to Section \ref{qla-pp}. 
Other formulations are obviously possible under milder assumptions if one applies 
previous sections. 
Hereafter, we will consider the case where $g^n(t,\theta)=g(t,\theta)$ and 
$K^n(t,s,\theta)=K(t,s,\theta)$ for simplicity of presentation. 

Consider the following conditions. 
\bd
\item[[H1\!\!]]$_{\bar{j}}$
$K(t,s,\theta)$ is an $\bbR_+^\sfd\otimes\bbR_+^{\sfd_0}$-valued  
$\calf\times\bbB(J)\times\bbB(\Theta)$-measurable 
function satisfying the following conditions. 
\bd
\im[(i)] 
For each $(t,\theta)\in I\times\Theta$, the process 
$[\hat{T}_0,t)\ni s\mapsto K(t,s,\theta)$ is 
$(\calf_s)_{s\in[\hat{T}_0,t)}$-optional. 

\im[(ii)] 
For each $(t,s)\in\bar{\bbN}\times J$, the mapping 
$\Theta\ni\theta\mapsto K(t,s,\theta)$ is $\bar{j}$ times differentialble a.s.,
%
%
and 
\beas 
\sum_{j=0}^{\bar{j}}
\sup_{(\omega,s,t,\theta)\in \Omega\times\bar{\bbN}\times J\times\Theta}|\partial_\theta^iK^n(t,s,\theta)|
&<&
\infty.
\eeas

\im[(iii)] 
For each $(t,\theta)\in I\times\Theta$, the mappings  
$[\hat{T}_0,t)\ni s\mapsto \partial_\theta^i K(t,s,\theta)$ ($i=0,1$)
are differentialble a.s., 
and 
${\colorb\text{ess.sup}_{\omega\in\Omega}\sup_{(t,s,\theta)\in J\times\Theta}}
| \partial_s\partial_\theta^i K(t,s,\theta)|<\infty$ for $i=0,1$. 
%
\ed
\ed

\begin{description}
\item[[H2\!\!]]$_{\bar{j}}$ 
For each $(\alpha)\in\cali$, 
$g^{\alpha}(t,\theta)$ is an 
nonnegative 
$\calf\times\bbB(I)\times\bbB(\Theta)$-measurable function 
for which the following conditions are fulfilled. 
\bd
\im[(i)] For each $(\alpha,\theta)\in\cali\times\Theta$, 
the process $(g^{\alpha}(t,\theta))_{t\in I}$ is predictable.  

\im[(ii)] For each $t\in I$, the mapping $\Theta\ni\theta\mapsto g(t,\theta)$ is 
$\bar{j}$ times differentiable a.s. and 
\beas 
\sum_{j=0}^{\bar{j}} \sup_{t\in I}\sup_{\theta\in\Theta}
\big\|(\partial_\theta)^j g(t,\theta)\big\|_p
&<&\infty
\eeas
for every $p>1$. 
\ed
\ed

\bd
\item[[H3\!\!]] 
For each $(\omega,n,\alpha,t,\theta)\in\Omega\times\bbN\times\cali\times I\times\Theta$, 
$\lambda^{n,\alpha}(t,\theta)=0$ if and only if $\lambda^{n,\alpha}(t,\theta^{\colorb *})=0$, 
and 
\beas 
\sup_{(n,t,\theta)\in I\times\Theta}\big\|\lambda^{n,\alpha}(t,\theta)^{-1}
1_{\{\lambda^{n,\alpha}(t,\theta)\not=0\}}
\big\|_p 
&<& \infty
\eeas
for every $p>1$ and $\alpha\in\cali$. 
\ed

Let 
\bea\label{lambda_infty0}
\lambda^\infty(t,\theta^*) 
&=& 
{\blue Gg}(\cdot,\theta^*)(t)
\eea
where $G=\sum_{m=0}^\infty K^{\blue m}(\cdot,\cdot,\theta^*)$ 
{\blue with  
\beas
K^m(t,s,\theta^*)
&=& 
\int_{\hat{T}_0}^t K(t,t_1,\theta^*)dt_1 
\int _{\hat{T}_0}^{t_1} K(t_1,t_2,\theta^*)dt_2
\\&&
\cdots
\int_{\hat{T}_0}^{t_{m-2}} K(t_{m-2},t_{m-1},\theta^*)dt_{m-1} K(t_{m-1},s,\theta^*). 
\eeas
}
Let
\bea\label{lambda_infty} 
\lambda^\infty(t,\theta)
&=&
g(t,\theta)+\int_{\hat{T}_0}^{t}K(t,s,\theta)\lambda^\infty(s,\theta^*)ds.
\eea
The representation (\ref{lambda_infty}) at $\theta=\theta^*$ is compatible with (\ref{lambda_infty0}). 
Define $\bbY$ and the index ${\colorb \chi_0}$ as before for the present $\lambda^\infty(t,\theta)$.

\bd
\item[[H4\!\!]]
For every $L>0$, there exists a constant $C_L$ such that 
\beas 
P[\chi_0<r^{-1}] &\leq& \frac{C_L}{r^L}
\eeas
for all $r>0$. 
\ed

\halflineskip

\begin{theorem*} 
Under Conditions {\colorr $[H1]_4$, $[H2]_4$, $[H3]$ and  $[H4]$,} the same results as Theorems \ref{260731-2} and \ref{260731-3} 
hold true. 
\end{theorem*}
\proof 
We need to verify [B3] in the present situation where $X^{n,\beta}=n^{-1}N^{n,\beta}$. 
We have 
\bea\label{heq} 
\lambda^{n}(t,\theta) 
&=& 
g(t,\theta)+\int_{\hat{T}_0}^{t-}K(t,s,\theta)n^{-1}dN^n_s
\nn\\&=&
g(t,\theta)+\int_{\hat{T}_0}^{t-}K(t,s,\theta)\lambda^n(s,\theta^*)ds
+\int_{\hat{T}_0}^{t-}K(t,s,\theta)n^{-1}d\tilde{N}^n_s
\eea
By $C_r$-inequality, 
\beas 
E[\lambda^n(t,\theta^{\blue *})^{2k}] 
&\lesssim&
E\big[g^n(t,\theta^{\blue *})^{2k}\big]
+\|K\|_\infty^{\blue 2k}\int_{\hat{T}_0}^t E[\lambda^n(s,\theta^{\blue *})^{2k}]ds
\\&&
+
E\l[\l(\int_{\hat{T}_0}^t{\blue K(t,s,\theta^{\blue *})}n^{-1}d\tilde{N}^n_s\r)^{2k}\r]
\eeas
for $k\in\bbN$. 
Then by an essentially the same inequality as (\ref{260809-1}) and by induction, we obtain 
\begin{en-text}
\beas 
E\big[\lambda^n(t,\theta)^{2k}\big] 
&\lesssim&
E\big[g^n(t,\theta)^{2k}\big]
+\sum_{j\leq k}\int_{\hat{T}_0}^t E[\lambda^n(s,\theta)^{2j}]ds
\\&\lesssim&
E\big[g^n(t,\theta)^{2k}\big]
+\int_{\hat{T}_0}^t E[\lambda^n(s,\theta)^{2k}]ds, 
\eeas
\end{en-text}
\beas 
E\big[\lambda^n(t,\theta^{\blue *})^{2^{\blue k}}\big] 
&\lesssim&
{\blue 1+}
E\big[g^n(t,\theta^{\blue *})^{2^{\blue k}}\big]
+\sum_{j\leq k}\int_{\hat{T}_0}^t E[\lambda^n(s,\theta^{\blue *})^{2^{\blue j}}]ds
\\&\lesssim&
{\blue 1+}
E\big[g^n(t,\theta^{\blue *})^{2^{\blue k}}\big]
+\int_{\hat{T}_0}^t E[\lambda^n(s,\theta^{\blue *})^{2^{\blue k}}]ds, 
\eeas
and hence by Gronwall's lemma, 
\bea\label{h09} 
\sup_{t\in I} E\big[\lambda^n(t,\theta^{\blue *})^p\big]  &<& \infty
\eea
for every $p>1$. 
Once again by the uniform version of the scheme of (\ref{260809-1}), 
with the aid of (\ref{h09}), 
we obtain 
\bea\label{h01} 
\sup_{t\in I}E\l[\l|\int_{\hat{T}_0}^tn^{-1/2}d\tilde{N}^n_s\r|^p\r]
&<& \infty
\eea
for every $p>1$. 

Now Equation (\ref{heq}) for $\theta=\theta^*$ gives 
\bea\label{h10}
\lambda^n(t,\theta^*) 
&=& 
G \bigg(g(\cdot,\theta^*)+\int_{\hat{T}_0}^{\cdot-}K(\cdot,s,\theta^*)n^{-1}d\tilde{N}^n_s\bigg)(t). 
\eea
Indeed, the convergence (\ref{h10}) holds in $\|\cdot\|_p$-norm uniformly in $t\in I$ 
due to [H2]$_0$ and (\ref{h01}). 
The limit should be 
\beas 
dX^\infty_t/dt &=& \lambda^\infty(t,\theta^*) 
\eeas
having a representation 
\bea\label{h11}
\lambda^\infty(t,\theta^*) &=& Gg(\cdot,\theta^*)(t)
\eea
deduced from (\ref{h10}). 
Comparing (\ref{h10}) and (\ref{h11}), we obtain 
\bea
n^{1/4}\sup_{t\in I}\|\lambda^n(t,\theta^*)-\lambda^\infty(t,\theta^*)\|_p
&\to&0
\eea
from (\ref{h01}). 
Since $X^n=n^{-1}N$ and 
\beas 
n^{1/4}\|{\colorb X^n_t}-X^\infty_t\|_p 
&\leq&
n^{1/4}\|n^{-1}{\colorb \tilde{N}}^n_t\|_p
+n^{1/4}\l\|\int_{\hat{T}_0}^{t-}(\lambda^n(s,\theta^*)-\lambda^\infty(s,\theta^*))ds\r\|_p, 
\eeas
we have 
$n^{1/4}\|{\colorb X}^n_t-X^\infty_t\|_p \to 0$, which gives [B3]. 
\qed

\begin{en-text}
{From} (\ref{heq}), we see 
\beas 
\lambda^\infty(t,\theta)
&=& 
g(t,\theta) +\int_{\hat{T}_0}^{t-}K(t,s,\theta)\lambda^\infty(s,\theta^*)ds
\eeas
\end{en-text}

\halflineskip
\halflineskip

In what follows, we shall discuss a two-dimensional Hawkes type process 
$N=(N_t)_{t\in[0,T]}$ 
as an illustrative example. 
%
Consider the parametric model of two-dimensional Hawkes process 
with intensity processes
\bea\label{270614-8}
\lambda^n(t,\theta)
&=&
g(t,\gamma)+\int_{{\colorr{\hat{T}_0}}}^{t-}e^{-b(t-s)}An^{-1}dN^n_s
\eea
with $\theta=(\gamma,b,A)$. 
It is remarked that in practice, we need $n\lambda^n(t,\theta)$ to make 
the function $\ell_n(\theta)$, and 
what is estimated is $ng$, not $g$, for the underlying intensity parameter. 
The asymptotics about the estimator of the parameter in $g$ 
is relative in the sense that its value can depends on the value of $n$ the user chooses. 
However it is rather natural because the baseline intensity is very changeable 
even interday and possibly randomly changing. 
There consistent estimation of the baseline intensity has no important meaning. 
We are rather interested in finding relations between $N^n$  and $X^n$, and then 
$g$ serves as a nuisance parameter. 
In statistical theory, 
similar treatments of scaling are found in change point problems and in volatility parameter estimation 
for small diffusions.

At the true values $\theta^*=(b^*,\gamma^*,A^*)$ of the parameters, 
it has two-dimensional intensity process 
expressed by 
\beas 
\lambda^n(t,\theta^*)
&=& 
g_t^*+\int^{t-}_{{\colorr \hat{T}_0}}A^* e^{-b^*(t-s)}n^{-1}dN^n_s,
\eeas
where $g^*_t=g(t,\gamma^*)$ is an $\bbR^2_+$-valued $C^1$-function, 
$A^*\in M_2(\bbR)=\bbR^2\otimes\bbR^2$ and $b^*\in(0,\infty)$. 
Let $C^*=A^*-b^*I$ for two-by-two identity matrix $I$. 
For a matrix $M$, let 
\beas 
G(M)_t 
&=&
e^{tM}\int_{{\colorr \hat{T}_0}}^te^{-sM}g^*_sds.
\eeas

\begin{lemma*}\label{rep_lambda}
\bea\label{lambda101} 
\lambda^\infty(t,\theta^*)
&=&
e^{{\colorr (t-\hat{T}_0)}C^*}g^*_{\hat{T}_0}
+e^{tC^*}\int_{{\colorr \hat{T}_0}}^te^{-sC^*}(\partial_sg^*_s+b^*g^*_s)ds
\>=\>
g^*_t+A^*G(C^*)_t.
\eea
In particular, if $g^*$ is a constant vector $g^*$, and if $C^*$ is invertible, then 
\bea\label{lambda102}
\lambda^\infty(t,\theta^*)
&=&
e^{{\colorr (t-\hat{T}_0)}C^*}(I+b^*C^{*-1})g^*-b^*C^{*-1}g^*.
\eea
\end{lemma*}
\proof 
Equation (\ref{lambda_infty}) becomes 
\beas
\lambda^\infty(t,\theta^*) 
&=&
g^*_t+\int_{{\colorr \hat{T}_0}}^te^{-b^*(t-s)}A^*\lambda^\infty(s,\theta^*)ds. 
\eeas
Therefore
\beas 
\partial_t\big(e^{-tC^*}\lambda^\infty(t,\theta^*)\big)
&=&
e^{-tC^*}\partial_tg^*_t-C^*e^{-tC^*}g^*_t
\\&&
-A^*e^{-tC^*}\int_{{\colorr \hat{T}_0}}^te^{-b^*(t-s)}A^*\lambda^\infty(s,\theta^*)ds.
+e^{-tC^*}A^*\lambda^\infty(t,\theta^*)
\\&=&
e^{-tC^*}\partial_tg^*_t-C^*e^{-tC^*}g^*_t+e^{-tC^*}A^*g^*_t
\\
&=&
e^{-tC^*}\partial_tg^*_t+b^*e^{-tC^*}g^*_t,
\eeas
which gives the first equality of (\ref{lambda101}). 
The second equality is due to integration-by-parts. 
Simple calculus gives (\ref{lambda102}) when $C^*$ is invertible and $g^*_t$ 
is constant. \qed
\halflineskip

{\colorb
\begin{lemma*}\label{270228-1}
Suppose that $b I + C^*$ is invertible. Then
\beas
\lambda^{\infty}(t,\theta)
&=&g(t,\gamma)+A(b-b^*)(b I + C^*)^{-1}G(-b I)_t +AA^*(b I+ C^*)^{-1}G(C^*)_t. \nonumber 
\eeas
\end{lemma*}
\proof 
{From} (\ref{lambda_infty}), we have 
\beas 
\lambda^\infty(t,\theta)
&=&
g(t,\gamma)+\int_{\hat{T}_0}^{t}e^{-b(t-s)}A\lambda^\infty(s,\theta^*)ds
\\&=&
g(t,\gamma)+\int_{\hat{T}_0}^{t} e^{-b (t-s)}A
\big\{g^*_s+A^*G(C^*)_s\big\}ds \nonumber 
\\&=&
g(t,\gamma)+AG(-b I)_t 
+ e^{-b t}\int^t_{{\colorr \hat{T}_0}} \int^t_u AA^* e^{s(b I+ C^*)}e^{-uC^*}g_u^*\>dsdu \nonumber
\\&=&
g(t,\gamma) +AG(-b I)_t + e^{-b t}AA^*\int^t_{\hat{T}_0}(b I+ C^*)^{-1}(e^{t(b I+ C^*)}-e^{u(b I+ C^*)})e^{-uC^*}
g_u^*\>du \nonumber 
\\&=&
g(t,\gamma)+AG(-b I)_t +AA^*(b I+ C^*)^{-1}(G(C^*)_t-G(-b I)_t) \nonumber \\
&=&
g(t,\gamma)+A(b-b^*)(b I + C^*)^{-1}G(-b I)_t +AA^*(b I+ C^*)^{-1}G(C^*)_t. 
\eeas
\qed

\begin{en-text}
\beas
\lambda^\infty(t,\theta)
&=&
g(t,\gamma)+\int_{\hat{T}_0}^{t}K(t,s,\theta)\lambda^\infty(s,\theta^*)ds.
\eeas
\begin{eqnarray}
\lambda^{\infty}(t,\theta)&=&g_t+\int^t_0 e^{-\beta (t-s)}A\lambda^{\infty}(s,\theta_{\ast})ds \nonumber \\
&=&g_t+\int^t_0 e^{-\beta (t-s)}A\{g_{s,\ast}+A_{\ast}e^{sC_{\ast}}\int^s_0e^{-uC_{\ast}}g_{u,\ast}du\}ds \nonumber \\
&=&
g(t,\gamma)+AG^{-\beta I}_t + e^{-\beta t}\int^t_0 \int^t_u AA_{\ast} e^{s(\beta I+ C_{\ast})}e^{-uC_{\ast}}g_{u,\ast}dsdu \nonumber \\
&=&
g_t +AG^{-\beta I}_t + e^{-\beta t}AA_{\ast}\int^t_0(\beta I+ C_{\ast})^{-1}(e^{t(\beta I+ C_{\ast})}-e^{u(\beta I+ C_{\ast})})e^{-uC_{\ast}}g_{u,\ast}du \nonumber \\
&=&g_t+AG^{-\beta I}_t +AA_{\ast}(\beta I+ C_{\ast})^{-1}(G^{C_{\ast}}_t-G^{-\beta I}_t) \nonumber \\
&=&g_t+A(\beta-\beta_{\ast})(\beta I + C_{\ast})^{-1}G^{-\beta I}_t +AA_{\ast}(\beta I+ C_{\ast})^{-1}G^{C_{\ast}}_t. \nonumber 
\end{eqnarray}
\end{en-text}

\begin{lemma*}\label{270302-1}
Let $\beta\in\bbR$, $v_1,v_2\in\mathbb{R}^2$ 
and let $B_1,B_2,C\in M_2(\mathbb{R})$. 
Suppose that ${\rm span}\{v_2,Cv_2\}=\mathbb{R}^2$,  
$\beta I+C$ is invertible, 
and that 
\begin{equation}\label{assumption-eq}
B_1e^{-\beta t}v_1+B_2e^{tC}v_2=0
\end{equation}
for any $t$ in an interval of $\bbR$. Then $B_2=O$ and $B_1v_1=0$.
\end{lemma*}
\proof
\begin{en-text}
Since $\text{span}(v_2,Cv_2)=\bbR^2$ and 
$e^{tC}=I_2+tC+o(t)$ as $t\to0$, 
there exists $t_0>0$ such that 
$
\text{span}\{v_2,e^{t_0C}v_2\}=\bbR^2
$ 
and hence $Bv_1=c_1v_2+c_2e^{t_0C}v_2$ for some $c_1,c_2\in\bbR$. 
Then by (\ref{assumption-eq}), 
\end{en-text}
The equation (\ref{assumption-eq}) holds for all $t\in\bbR$. 
Differentiating $B_1v_1+B_2e^{t(\beta I+C)}v_2=0$ in $t$, we see  
$B_2(\beta I+C)e^{-t(\beta I+C)}v_2=0$ for all $t$. 
Since $\text{span}\{v_2,e^{-t(\beta I+C)}v_2\}=\bbR^2$ for small $t>0$, 
$B_2(\beta I+C)=O$. Therefore $B_2=O$ by the invertibility of $\beta I+C$. 
This entails $B_1v_1=0$ from (\ref{assumption-eq}). 
\qed
\begin{remark*}\rm 
The claim of Lemma \ref{270302-1} is not valid without the invertibility of $\beta I_2+C$. 
For example, 
the assumptions are satisfied for  
\beas 
C=\l[\begin{array}{cc}-1&0\\0&1\end{array}\r],\quad 
B_1=\l[\begin{array}{cc}1&0\\0&1\end{array}\r],\quad 
B_2=\l[\begin{array}{cc}-1&0\\0&0\end{array}\r],\quad 
v_1=\l[\begin{array}{c}1\\0\end{array}\r],\quad 
v_2=\l[\begin{array}{c}1\\1\end{array}\r]
\eeas
and $\beta=1$. 
However, $B_2\not=O$ and $B_1v_1\not=0$. 
\end{remark*}

Suppose that $g^*_t=g(t,\gamma^*)$ is an $\bbR^2$-valued polynomial in $t$ of degree $p$. 
Then, for any $M\in\text{GL}(\bbR^2)$, there exist 
$\bbR^2$-valued smooth (in $M$) mappings $c_\ell(M)$ such that 
\beas 
\partial_s \bigg(\sum_{\ell=0}^p (s-\hat{T}_0)^\ell e^{-(s-\hat{T}_0)M}c_\ell(M)\bigg)
&=& 
e^{-(s-\hat{T}_0)M}g^*_s.
\eeas
Then 
\bea\label{270614-2} 
G(M)_t &=& 
\sum_{\ell=0}^p (t-\hat{T}_0)^\ell c_\ell(M) -e^{(t-\hat{T}_0)M}c_0(M). 
\eea
In particular, 
\beas 
g^*_{\hat{T}_0} &=& c_1(M)-Mc_0(M)
\eeas
and
\beas 
\partial_t^{p+1}G(M)_t &=& 
-M^{p+1}e^{(t-\hat{T}_0)M}c_0(M). 
\eeas

Suppose that $g(t,\gamma)$ is an $\bbR^2$-valued polynomial in $t$ of degree $p$ and 
that 
\bea\label{identify-1}
\lambda^\infty(t,\theta) = \lambda^\infty(t,\theta^*)
\quad (\forall t\in I)
\eea
Differentiating $\lambda^\infty(t,\theta)-\lambda^\infty(t,\theta^*)\equiv0$ $(p+1)$-times in $t$ 
with the expression of $\lambda^\infty(t,\theta)$ given in Lemma \ref{270228-1}, 
we obtain 
\beas 
A(b-b^*)(bI+C^*)^{-1}(-1)^{p+2}b^{p+1}e^{-(t-\hat{T}_0)bI}c_0(-bI)
\\
-(C-C^*)A^*(bI+C^*)^{-1}(C^*)^{p+1}e^{(t-\hat{T}_0)C^*}c_0(C^*)
&=&
0
\eeas
for all $t$. 

Let us consider the following condition. 
\begin{description}
\im[[M\!\!]] 
The parametric model admits a continuous extension to $\overline{\Theta}$ and 
that the following conditions are satisfied on $\overline{\Theta}$\ : 
\bi
\im[(i)]  $b\not=0$, 
\im[(ii)] the matrices $C^*$ and $bI+C^*$ are invertible, 
\im[(iii)] $c_0(-bI)\not=0$, 
\im[(iv)] $c_0(C^*)$ is not an eigenvector of $C^*$. 
\im[(v)] 
$g(t,\gamma)=\sum_{\ell=0}^p a_\ell(\gamma)t^\ell$, and 
$g$ is identifiable, i.e., 
$\gamma=\gamma^*$ if 
$g(t,\gamma)=g(t,\gamma^*)$ for all $t\in I$. 
\im[(vi)] $\inf_{t\in I,\>\theta\in\overline{\Theta}}g(t,\gamma)>0$. 
\im[(vii)] 
The mapping $\gamma\mapsto g(t,\gamma)$ is of class $C^4$ with each derivative 
admitting a continuous extension to $\overline{\Theta}$. 
Moreover 
$\sum_{\alpha=1}^2(\partial_\gamma g^\alpha)^{\otimes2}(t,\gamma^*)$ 
is positive definite for some $t\in I$. 
\ei
\end{description}
\halflineskip

\begin{lemma*}\label{270614-4}
Under $[M]$, 
the matrix $\Gamma$ is positive definite.
\end{lemma*}

\begin{proof}
Let $x$ satisfy $x^{\top}\Gamma x=0$. Then it is sufficient to show $x=0$. 
Condition [M] (vi) in particular ensures $\inf_{t\in I}g(t,\gamma^*)>0$, therefore 
$\inf_{t\in I}\lambda^{\infty,\alpha}(t,\theta^*)>0$ for $\alpha=1,2$. 
Thus $\Gamma$ of (\ref{270614-1}) is well defined and we obtain 
\begin{equation}\label{lambda-eq}
\partial_{\theta}\lambda^{\infty,\alpha}(t,\theta^{\ast})\cdot x=0
\end{equation}
for any $t,\alpha$.
Simple calculations with the formula in Lemma \ref{270228-1} show
\begin{equation*}
\partial_{\gamma}\lambda^{\infty}(t,\theta^{\ast})=\partial_{\gamma}g(t,\gamma^{\ast}), 
\quad \partial_b\lambda^{\infty}(t,\theta_{\ast})=G(-b_{\ast}I)_t-G(C^{\ast})_t,
\end{equation*}
\begin{equation*}
(\partial_{A_{\alpha' 1}}\lambda^{\infty,\alpha}, \partial_{A_{\alpha' 2}}\lambda^{\infty,\alpha})^{\top}(t,\theta^{\ast})=\delta_{\alpha,\alpha'}G(C^{\ast})_t
\end{equation*}
for $1\leq \alpha,\alpha'\leq 2$.
Rewriting (\ref{lambda-eq}) with these expressions, we have
\begin{equation*}
\partial_{\gamma}g(t,\gamma^{\ast})\cdot x_{\gamma}+\left(
\begin{array}{ll}
x_{A_{11}} & x_{A_{12}} \\
x_{A_{21}} & x_{A_{22}}
\end{array}
\right)G(C^{\ast})_t+x_b((G(-b_{\ast}I)_t-G(C^{\ast})_t)=0
\end{equation*}
for any $t$, where $x=(x_b,(x_{A_{\alpha,\alpha'}})_{\alpha,\alpha'},x_\gamma)$.
Moreover, by differentiating the both sides of the above equation $p+1$ times with respect to $t$, we obtain
\begin{equation*}
\bigg(x_bI-\left(
\begin{array}{ll}
x_{A_{11}} & x_{A_{12}} \\
x_{A_{21}} & x_{A_{22}}
\end{array}
\right)\bigg)(C^{\ast})^{p+1}e^{(t-\hat{T}_0)C^{\ast}}c_0(C^{\ast})+x_b(-1)^pb^{p+1}e^{-b^{\ast}(t-\hat{T}_0)}c_0(-b^{\ast}I)=0
\end{equation*}
for any $t$. 
Here the expression (\ref{270614-2}) was used. 

Now Lemma \ref{270302-1} and $[M]$ (i)-(iv) yield $x_b=0$ and $x_{A_{\alpha,\alpha'}}=0$ for $1\leq \alpha,\alpha'\leq 2$.
Therefore $\partial_{\gamma}g(t,\gamma^{\ast})\cdot x_{\gamma}=0$ for any $t$ by (\ref{lambda-eq}).
Then [M] (vii) implies $x_{\gamma}=0$. 
\end{proof}

\begin{theorem*}\label{270614-3}
Under Condition $[M]$, the same results as Theorems \ref{260731-2} and \ref{260731-3} 
hold for the Hawkes type model (\ref{270614-8}). 
\end{theorem*}
\proof 
Under [M] (i)-(v), 
Lemma \ref{270302-1} implies that $C-C^*=O$ and that 
$b-b^*=0$ since $A=bI+C=bI+C^*$ is invertible and $c_0(bI)\not=0$ 
as well as $b\not=0$. 
Therefore $g(t,\gamma)=g(t,\gamma^*)$ from (\ref{identify-1}), which entails  
$\gamma=\gamma^*$. 

Since $\overline{\Theta}$ is compact and $\bbY$ has a continuous extension to it, 
we see from Lemma \ref{270614-4} that  
there exists a positive canstant $c$ such that 
$\bbY(\theta)\leq -c\>|\theta-\theta^*|^2$ for all $\theta\in\Theta$. 
Therefore Condition [H4] is satisfied, obviously.

\begin{example*}\rm 
Let us consider a mapping 
\beas 
g(t,\gamma)&=&
{\bm \gamma}_1(t-T^*)^2+{\bm \gamma}_2
\eeas
taking values in $(0,\infty)^2$ for $t\in [T_0,T_1]$, where 
$T^*=(T_1+T_0)/2$ and ${\bm \gamma}_i$ ($i=1,2$) are parameters in $(0,\infty)^2$. 
By definition, 
\beas 
c_2(M) &=& -M^{-1}{\bm \gamma}_1^*
\\
c_1(M) &=& 
2(T^*-\hat{T}_0)M^{-1}{\bm \gamma}_1^*-2M^{-2}{\bm \gamma}_1^*
\\
c_0(M) &=&
-(T^*-\hat{T}_0)^2M^{-1}{\bm \gamma}_1^*
-M^{-1}{\bm \gamma}_2^*
+2(T^*-\hat{T}_0)M^{-2}{\bm \gamma}_1^*-2M^{-3}{\bm \gamma}_1^*.
\eeas
Then, for [M] (iii),  $c_0(-bI)\in(0,\infty)^2$ whenever $\beta>0$. 
Condition [M] (iv) requires that 
\beas 
c_0(C^*) &=&
\big\{-(T^*-\hat{T}_0)^2(C^*)^{-1}
+2(T^*-\hat{T}_0)(C^*)^{-2}-2(C^*)^{-3}\big\}{\bm \gamma}_1^*
-M^{-1}{\bm \gamma}_2^*
\eeas
is not an eigenvector of $C^*$. 
\end{example*}

\begin{en-text}
We shall consider three functions for $g(t,\gamma)$ as follows: 
\begin{description}
\im[[M1\!\!]] $g(t,\gamma)=\l[\begin{array}{c}\gamma_1\\\gamma_2\end{array}\r]$ 
for $\gamma=(\gamma_1,\gamma_2)\in(0,\infty)^2$ 
\im[[M2\!\!]] $g(t,\gamma)=\l[\begin{array}{c}\gamma_1t+\gamma_3\\\gamma_2t+\gamma_4
\end{array}\r]
$ 
for $\gamma=(\gamma_1,\gamma_2,\gamma_3,\gamma_4)\in\bbR^4$ 
such that $\inf_{t\in[T_0,T_1]}g(t,\gamma)>0$. 
\im[[M3\!\!]] $g(t,\gamma)=\l[\begin{array}{c}\gamma_1t^2+\gamma_3t+\gamma_5
\\\gamma_2t^2+\gamma_4t+\gamma_6
\end{array}\r]
$ 
for $\gamma=(\gamma_1,\gamma_2,\gamma_3,\gamma_4,\gamma_5,\gamma_6)\in\bbR^6$ 
such that $\inf_{t\in[T_0,T_1]}g(t,\gamma)>0$. 

\end{description}
The following lemma ensures identifiability of these models. 
\begin{lemma*}\label{hderiv}
\bd
\im[(i)] Given $g(t,\gamma)$ by $[M1]$, suppose that 
$\gamma^*$ is not an eigenvector of $A^*$ and that 
$\beta I_2+C^*$ is invertible. 
Then $\beta=\beta^*$ and $C=C^*$. 

\ed
\end{lemma*}
\end{en-text}

}

\begin{en-text}
\koko
{\colorr 
\bi
\im Write a theorem for $g(t,\gamma)$ 
\im Check the conditions for a quadratic function $g(t,\gamma)$. 
\ei
}

\koko
Denote by $g_*$, $A_*$ and $\beta_*$ the true parameters of 
$g$, $A$ and $\beta$, respectively.

\koko 
\newpage
Verify [H4] here !

\begin{example*}
For $A\in \bbR^\sfd\otimes\bbR^\sfd$, set $C=A-\beta I$, $I$ being the unit matrix. 
Let $g\in\bbR^\sfd$. 
Suppose that $C$ is invertible. 
Define $\lambda:\bbR_+\to \bbR^\sfd\otimes\bbR^\sfd$ by 
\beas 
\lambda_t &=& e^{tC}(I+\beta C^{-1})g-\beta C^{-1}g. 
\eeas
Then $\lambda$ solves 
\beas 
\lambda_t &=& g + \int_0^t e^{-\beta(t-s)}A\lambda_sds. 
\eeas

In our example, 
\bea\label{ip-1}
\lambda^{\infty}(t,\theta) 
& =& 
g + A\int_{\hat{T}_0}^t e^{-\beta(t-s)}\lambda^\infty(s,\theta^*)ds
\eea
and $\lambda^\infty(\cdot,\theta^*)$ solves the equation 
\beas 
\lambda^{\infty}(t,\theta^*) 
& =& 
g^* + A^*\int_{\hat{T}_0}^t e^{-\beta^*(t-s)}\lambda^\infty(s,\theta^*)ds.
\eeas
Namely, 
\bea\label{ip-2}
\lambda^\infty(t,\theta^*) &=& e^{tC^*}(I+\beta^* C^{*-1})g^*-\beta C^{*-1}g^*
\eea
if $C^*$ is invertible. 

We consider 
Identifiability problem: 
\bea\label{ip-3}
\lambda^\infty(t,\theta)=\lambda^\infty(t,\theta^*)\quad(\forall t\in I)
&\Iku^?& 
g=g^*,\ \beta=\beta^*,\ A=A^*
\eea
{From} (\ref{ip-3}), we have $g=g^*$. 
Then (\ref{ip-3}) is equivalent to the question 
\bea\label{ip-4}
A\int_{\hat{T}_0}^t e^{-\beta(t-s)}\lambda^\infty(s,\theta^*)ds
=A^*\int_{\hat{T}_0}^t e^{-\beta^*(t-s)}\lambda^\infty(s,\theta^*)ds\quad(\forall t\in I)
&\Iku^?& 
\beta=\beta^*,\ A=A^*
\eea
Let $P^*$ be the projection on $\bbR^\sfd$ to $L[g^*]$, and $P^\perp=I-P^*$. 
For simplicity, suppose $\sfd=2$. 
Write $A=[c_{11}g^*+c_{12}g^\perp,c_{21}g^*+c_{22}g^\perp]'$, where $g^\perp$ is a unit vector orthogonal to $g^*$. 
Accordingly, $A^*$ is given by $A^*=[c_{11}^*g^*+c_{12}^*g^\perp,c_{21}^*g^*+c_{22}^*g^\perp]'$. 
Since $C^*$ commutes with $A^*$, by (\ref{ip-2})
\beas
A^*\lambda^\infty(t,\theta^*) &=& e^{tC^*}(I+\beta^* C^{*-1})A^*g^*-\beta C^{*-1}A^*g^*
\\&=& 
\bigg(e^{tC^*}(I+\beta^* C^{*-1})-\beta C^{*-1}\bigg)
\l[\begin{array}{c}c_{11}^* \\ c_{21}^* \end{array}\r]
|g^*|^2
\eeas

{\colorrz If $g^*$ is an eigen vector of $A^*$, $A^*g^*=\lambda^* g^*$, .... $A$ has still three parameters
\beas
A^*\lambda^\infty(t,\theta^*) &=& e^{tC^*}(I+\beta^* C^{*-1})A^*g^*-\beta C^{*-1}A^*g^*
\\&=& 
\bigg(e^{t(\lambda^*-\beta^*)}(I+\beta^*/\lambda^*)-\beta /\lambda^*\bigg)
\lambda^*g^*
\eeas
In $\sfd=2$, we only have freedom 
$g^*$, $\lambda^*$ and $\beta^*$, for the family $\{\lambda^\infty(\cdot,\theta^*)\}_{\theta^*\in\Theta}$. 

}

\end{example*}

\end{en-text}

\section{{\colorb Appendix}}
{\colorb 
\begin{en-text}
The family 
\beas 
\bigg\{b_n^{-1}\int_{T_0}^{T_1}\int |c^n_\alpha(t,x)|^2\mu^{n,\alpha}(dt,dx)\bigg\}_{n\in\bbN}
\eeas
is tight. 
\im[{\rm (iii)}] As $n\to\infty$, 
\beas 
\sup_{t\in[T_0,T_1]}\int |c^n_\alpha(t,x)|\mu^{n,\alpha}(\{t\},dx) &\to^p& 0
\eeas
\end{en-text}
{\it Proof of Lemma \ref{270730-1}.} 
In the present situation, 
\beas 
\langle L^n,L^n\rangle 
&=& 
\sum_\alpha\int_{T_0}^\cdot\int_{\sf E_\alpha}c^{n,i}_\alpha(s,x)^{\otimes2}
\nu^{n,\alpha}(ds,dx). 
\eeas
Therefore, $\langle L^n,L^n\rangle_t\to^p\int_{T_0}^tg_s^{\otimes2}ds$ for each $t$.

Let ${\sf M}$ be an $\F$-locally square-integrable purely discontinous martingale. 
For $\ep>0$, there is a canonical representation 
\beas 
{\sf M}_t &=& B^\ep_t+{\sf M}^{\overline{\ep}}_t
+{\sf M}^{\underline{\ep}}_t
\eeas
\begin{en-text}
\beas 
{\sf M}_t &=& B^\ep_t+\int_{T_0}^t \int z1_{\{|z|\leq\ep\}}\>(\mu-\nu)(ds,dz)
+\int_{T_0}^tz1_{\{|z|>\ep\}}\mu(ds,dx)
\eeas
\end{en-text}
with 
$
{\sf M}^{\underline{\ep}}_t =\int_{T_0}^t \int z1_{\{|z|>\ep\}}\>\mu(ds,dz)
$ 
and 
$
{\sf M}^{\overline{\ep}}_t = \int_{T_0}^t \int z1_{\{|z|\leq\ep\}}\>(\mu-\nu)(ds,dz)
$, 
\begin{en-text}
\beas 
{\sf M}^{\underline{\ep}} &=& B^\ep_t + \int_{T_0}^t \int z1_{\{\ep<|z|\leq1\}}\>(\mu-\nu)(ds,dz)
\eeas
and 
\beas 
{\sf M}^{\overline{\ep}} &=& \int_{T_0}^t \int z1_{\{|z|\leq\ep\}}\>(\mu-\nu)(ds,dz).
\eeas
\end{en-text}
where $B^\ep_t$ is a predictable bounded variational process and $\mu=\mu^M$ is the integer-valued random measure of jumps of $M$ and compensated by $\nu$. 

Since a sequence of predictable times exhausts the jumps of $B^\ep$ and 
$\Delta L^n_T=0$ for any predictable time $T$ for the quasi-left continuous $L^n=(L^{n,i})_i$, 
we have $[L^{n,i},B^\ep]=0$. 
Then the variation 
\bea\label{270906-1}
{\rm Var}\l[L^{n,i},{\sf M}\r]
&=&
{\rm Var}\l[L^{n,i},{\sf M}^{\underline{\ep}}\r]
+{\rm Var}\l[L^{n,i},{\sf M}^{\overline{\ep}}\r]
\nn\\&\leq&
\sum_{s\leq\cdot}|\Delta L^{n,i}_s||\Delta{\sf M}^{\underline{\ep}}_s|
\>1_{\{|\Delta{\sf M}^{\underline{\ep}}_s|>\ep\}}
+[L^{n,i},L^{n,i}]^{1/2}\l[{\sf M}^{\overline{\ep}},{\sf M}^{\overline{\ep}}\r]^{1/2}
\nn\\&\leq&
\sup_{s\leq T_1}|\Delta L^{n,i}_s|\>
\l(2\sup_{s\leq T_1} |{\sf M}_s|\r)\>\#\{s\leq T_1;\>|\Delta{\sf M}^{\underline{\ep}}_s|>\ep\}
\nn\\&&
+[L^{n,i},L^{n,i}]^{1/2}\l[{\sf M}^{\overline{\ep}},{\sf M}^{\overline{\ep}}\r]_{T_1}^{1/2}.
\eea

For $\eta>0$, 
\beas 
\sup_s|\Delta L^n_s|^2
&\leq&
\eta^2+\sup_s\l(|\Delta L^n_s|^21_{\{|\Delta L^n_s|\geq \eta\}}\r)
\\&\leq&
\eta^2+\sum_\alpha\int_{T_0}^{T_1}|c^n_\alpha(s,x)|^2
1_{\{|c^n_\alpha(s,x)|\geq \eta\}}
(\mu^{n,\alpha}+\nu^{n,\alpha})(ds,dx)
\eeas
Thus, the Lenglart inequality gives 
$\sup_s|\Delta L^n_s|\to^p0$ as $n\to\infty$ under (iii). 

Moreover, the family 
\beas
[L^{n,i},L^{n,i}]_{T_1} &=& 
\sum_\alpha\int_{T_0}^{T_1}\int |c^{n,\alpha}(s,x)|^2\mu^{n,\alpha}(ds,dx)
\qquad (n\in\bbN)
\eeas
is stochastically bounded due to (ii) and the Lenglart inequality. 
By definition, $\lim_{\ep\down0}\l[{\sf M}^{\overline{\ep}},{\sf M}^{\overline{\ep}}\r]_{T_1}=0$ a.s. 
These properties applied to (\ref{270906-1}) entail 
\bea\label{270906-2}
\text{Var}[L^{n,i},{\sf M}]_{T_1}\to^p0
\eea
as $n\to\infty$. 

Since {\colorc the signed measure} 
$\langle L^{n,i},{\sf M}\rangle <\!< \langle {\sf M},{\sf M}\rangle$, 
there is a predictable process $d\langle L^{n,i},{\sf M}\rangle/d\langle {\sf M},{\sf M}\rangle(t)$ 
that is a version of the Radon-Nikodym derivative. 
Define an increasing process $\langle L^{n,i},{\sf M}\rangle^+$ by 
\beas 
\langle L^{n,i},{\sf M}\rangle^+_t 
&=& 
\int_{T_0}^t 
1_{\l\{\frac{d\langle L^{n,i},{\sf M}\rangle}{d\langle {\sf M},{\sf M}\rangle}(s)\geq0\r\}}
d\langle L^{n,i},{\sf M}\rangle_s
\>\equiv\>
\int_{T_0}^t 
1_{\l\{\frac{d\langle L^{n,i},{\sf M}\rangle}{d\langle {\sf M},{\sf M}\rangle}(s)\geq0\r\}}
\frac{d\langle L^{n,i},{\sf M}\rangle}{d\langle {\sf M},{\sf M}\rangle}(s)
d\langle {\sf M},{\sf M}\rangle_s. 
\eeas
The process $\langle L^{n,i},{\sf M}\rangle^+$ is L-dominated by the optional process
$\text{Var}[L^{n,i},{\sf M}]$. 
\begin{en-text}
\beas 
\int_{T_0}^\cdot 
1_{\l\{\frac{d\langle L^{n,i},{\sf M}\rangle}{d\langle {\sf M},{\sf M}\rangle}(s)\geq0\r\}}
d\text{Var}[ L^{n,i},{\sf M}]_s. 
\eeas
\end{en-text}
In fact, there exists a reducing sequence $\tau_m$ of stopping times for which 
both $\text{Var}\langle L^{n,i},{\sf M}\rangle$ and $\text{Var}[ L^{n,i},{\sf M}]$ are integrable. 
Then for any stopping time $\tau$, 
\beas 
E[\langle L^{n,i},{\sf M}\rangle^+_{\tau_m\wedge \tau}]
&=&
E\l[\int_{T_0}^{\tau_m\wedge \tau}
1_{\l\{\frac{d\langle L^{n,i},{\sf M}\rangle}{d\langle {\sf M},{\sf M}\rangle}(s)\geq0\r\}}
d[ L^{n,i},{\sf M}]_s\r]
\\&\leq&
E\l[\int_{T_0}^{\tau_m\wedge \tau}
1_{\l\{\frac{d\langle L^{n,i},{\sf M}\rangle}{d\langle {\sf M},{\sf M}\rangle}(s)\geq0\r\}}
d\text{Var}[ L^{n,i},{\sf M}]_s\r]
\\&\leq&
E\big[\text{Var}[ L^{n,i},{\sf M}]_{\tau_m\wedge \tau}\big]
\eeas
and let $m\nearrow\infty$ to obtain the L-domination. 
Since we know that 
$\{\sup_s\Delta \text{Var}[ L^{n,i},{\sf M}]_s\}_n$ is stochastically bounded, 
we obtain $\langle L^{n,i},{\sf M}\rangle^+_t\to^p0$ as $n\to\infty$ 
from (\ref{270906-2}) by the Lenglart inequality. 
In the same fashion, we can prove 
\beas 
\langle L^{n,i},{\sf M}\rangle^-_t
&:=&
-\int_{T_0}^t 
1_{\l\{\frac{d\langle L^{n,i},{\sf M}\rangle}{d\langle {\sf M},{\sf M}\rangle}(s)<0\r\}}
d\langle L^{n,i},{\sf M}\rangle_s
\>\to^p0
\eeas
Thus we obtained $\langle L^{n,i},{\sf M}\rangle_t\to^p0$. 

\begin{en-text}
Therefore 
${\rm Var}[n^{-1/2}\tilde{N}^{n,\alpha},{\sf M}]_{T_1} \to^p0$, and hence 
${\rm Var}[L^n,{\sf M}]_{T_1} \to^p0$, consequently 
$\langle L^n,{\sf M}\rangle_t \to^p0$
 as $n\to\infty$ for every $t$. 
\end{en-text}
 Obviously,  
$\langle L^n,{\sf M}\rangle=0$ for $\F$-continuous martingales ${\sf M}$. 
Under the properties shown above, the result follows from Jacod's theorem. 
\qed
}


\bibliographystyle{spmpsci}      
\bibliography{bibtex-20151019-20151106}   

\end{document}